\documentclass[10pt,a4paper]{amsart}
\usepackage[a4paper,margin=27mm,headheight=15pt]{geometry}
\usepackage[T1]{fontenc}
\usepackage{lmodern,microtype}
\usepackage{amsmath,amssymb,amsthm,mathtools,mathrsfs}
\usepackage{enumitem,booktabs,array,longtable}
\usepackage{xcolor}
\usepackage[colorlinks=true,linkcolor=blue!45!black,citecolor=blue!45!black,urlcolor=blue!45!black]{hyperref}
\hypersetup{
 pdftitle={Lawson--Deligne Classes and Applications},
 pdfauthor={Wenchuan Hu},
 pdfsubject={Filtered-current construction, geometric applications, and a precise Bloch--KLM comparison},
 pdfkeywords={Lawson homology, Deligne cohomology, Abel-Jacobi, Picard-Fuchs, blow-up}
}
\usepackage{etoolbox}
\makeatletter
\patchcmd{\l@section}{\addvspace{1.0em \@plus\p@}}{\addvspace{0pt}}{}{}
\makeatother
\usepackage{fancyhdr}
\numberwithin{equation}{section}
\newtheorem{theorem}{Theorem}[section]
\newtheorem{proposition}[theorem]{Proposition}
\newtheorem{lemma}[theorem]{Lemma}
\newtheorem{corollary}[theorem]{Corollary}
\theoremstyle{definition}

\newtheorem{example}[theorem]{Example}
\newtheorem{remark}[theorem]{Remark}
\newcommand{\C}{\mathbb C}
\newcommand{\R}{\mathbb R}
\newcommand{\Q}{\mathbb Q}
\newcommand{\Z}{\mathbb Z}
\newcommand{\PP}{\mathbb P}
\newcommand{\HH}{\mathfrak H} % Upper half-plane, defined in Section 1.2.
\newcommand{\CH}{\operatorname{CH}}

\newcommand{\Pic}{\operatorname{Pic}}
\newcommand{\PD}{\operatorname{PD}}
\newcommand{\im}{\operatorname{im}}
\newcommand{\Cone}{\operatorname{Cone}}
\newcommand{\Bl}{\operatorname{Bl}}
\newcommand{\Reg}{\mathcal R}

\newcommand{\Dcur}{{\mathcal D'}}
\newcommand{\Icur}{\mathcal I}
\newcommand{\Zcyc}{\mathcal Z}
\newcommand{\homol}{\mathrm{hom}}

\newcommand{\HD}{H_{\mathcal D}}

\newcommand{\Img}{\mathscr I}
\newcommand{\LoopIm}{\mathscr W}
\newcommand{\clD}{\operatorname{cl}_{\mathcal D}}
\newcommand{\ahat}{\widehat a}

\title{Lawson--Deligne Classes and Applications}
\author{Wenchuan Hu}
\address{School of Mathematics, Sichuan University, Chengdu 610064, China}
\email{huwenchuan@gmail.com, wenchuan@scu.edu.cn}
\date{\today}



\begin{document}
\begin{abstract}
We construct the integral Lawson--Deligne map of weight $q=n-p-k-1$ on smooth complex projective $n$-folds using filtered currents. It lifts the Friedlander--Mazur cycle class, recovers the reduced generalized Abel--Jacobi invariant on homologically trivial classes, and is compatible with algebraic correspondences. 

A Picard--Fuchs separation argument applied to the conic and van Geemen normal functions on the mirror quintic determines explicit regulator subspaces modulo the full rational period group. For prescribed elliptic moduli and a suitable mirror-quintic fiber, the subspace generated by their $a$- and $b$-loop products has dimension twice the $\Q$-dimension of the period-monomial space. Moduli $i\sqrt{\ell_j}$ for distinct primes $\ell_j$ give $2^{k+1}$ independent images on varieties of dimension $p+k+2$; one repeated imaginary quadratic modulus gives dimension four for every $k\geq1$. Compatibility with known projective-bundle and blow-up decompositions yields independent exceptional subspaces on smooth rational varieties. 

We also compare the higher Chow composite with the Bloch--KLM regulator after lowering the Hodge filtration. The KLM representative reduces to a cut-current class, and equality with the Lawson composite is proved in degree zero and for constant-unit decomposable classes. The general positive-degree comparison is reduced to an explicit filtered-realization condition.
\end{abstract}
\maketitle
\begingroup
\small\setlength{\parskip}{0pt}
\tableofcontents
\endgroup

\section{Introduction and principal results}
Let $X$ be a smooth complex projective variety of dimension $n$. We use dimension grading for algebraic cycles and write
\[
 L_pH_{2p+k}(X)=\pi_k\Zcyc_p(X),\qquad p,k\geq0,
\]
where $\Zcyc_p(X)$ is the topological group of algebraic $p$-cycles with its Chow topology and zero basepoint. We set $L_pH_{2p+k}(X,\Q)=L_pH_{2p+k}(X)\otimes_{\Z}\Q$. Write
\[
 \Phi_{p,2p+k}:L_pH_{2p+k}(X)\longrightarrow H_{2p+k}(X,\Z),
 \qquad L_pH_{2p+k}(X)_{\homol}=\ker\Phi_{p,2p+k}.
\]
The same subscript denotes the kernel of the rational cycle map; by flatness of $\Q$, this is $L_pH_{2p+k}(X)_{\homol}\otimes\Q$. The trace construction and this cycle map are the established constructions of Friedlander and Friedlander--Mazur \cite{F,FM}; the secondary invariant on the kernel is related to the author's generalized Abel--Jacobi map \cite{HuAJ}.

Throughout the paper the indices are
\begin{equation}\label{eq:indices}
 c=n-p,\qquad M=2c-k,\qquad q=c-k-1.
\end{equation}
We write $\ahat^X_{p,k}$ for the integral map and $\Reg_{p,k}^{X}$ for its rational-coefficient version. These denote the same construction with different coefficients. Unless otherwise stated, the secondary applications have $q\geq1$.

\subsection{The construction}
\begin{theorem}[Integral construction and rational realization]\label{thm:mainmap}
For $p,k\geq0$ and $q\geq1$, there are homomorphisms
\begin{equation}\label{eq:mainmap}
 \begin{aligned}
 \ahat^X_{p,k} &: L_pH_{2p+k}(X)\longrightarrow\HD^{M}(X,\Z(q)),\\
 \Reg_{p,k}^{X} &: L_pH_{2p+k}(X,\Q)\longrightarrow\HD^{M}(X,\Q(q)).
 \end{aligned}
\end{equation}
They commute with change of coefficients and satisfy $\delta_2\ahat^X_{p,k}=\PD\Phi_{p,2p+k}$, and similarly over $\Q$, where $\PD$ denotes Poincar\'e duality. In the normalized filtered-current cone a trace representative gives
\[
 \ahat^X_{p,k}([f])=[(R_f,R_f,0)].
\]
The rational homologically trivial part takes values in
\begin{equation}\label{eq:Jintro}
 J_q^{M-1}(X;\Q)
 =\frac{H^{M-1}(X,\C)}{F^qH^{M-1}(X,\C)+H^{M-1}(X,\Q)}.
\end{equation}
Under the current-to-chain convention of Proposition~\ref{prop:AJ}, this restriction is the reduced generalized Abel--Jacobi invariant. The maps are compatible with proper pushforward, smooth pullback, and ordinary algebraic Deligne cycle operations. In particular the rational map is compatible with algebraic correspondences, with the degree and weight shifts of Theorem~\ref{thm:corr}.
\end{theorem}
For $q\leq0$ we explicitly use the cone convention of Section~\ref{sec:cone}; in this convention only the Betti class remains.

\subsection{Periods, prescribed elliptic factors, and exact subspace ranks}\label{subsec:ellipticparameters}
\paragraph{Elliptic-curve parameter space.}
We write
\[
 \HH=\{\tau\in\C:\operatorname{Im}\tau>0\},\qquad
 \HH^k=\underbrace{\HH\times\cdots\times\HH}_{k\text{ factors}},
 \qquad \HH^0=\{*\}.
\]
Thus $\HH$ is the complex upper half-plane and $\HH^k$ its $k$-fold Cartesian product. Put
\[
 E_\tau=\C/(\Z+\tau\Z),\qquad
 T_{\boldsymbol\tau}=E_{\tau_1}\times\cdots\times E_{\tau_k}.
\]
The ordered lattice basis $(1,\tau_j)$ marks the $a$- and $b$-loops. For a fixed tuple define its \emph{period coefficient space}
\begin{equation}\label{eq:Sintro}
 S_{\boldsymbol\tau}
 =\operatorname{span}_{\Q}\left\{\prod_{i\in I}\tau_i:I\subseteq\{1,\ldots,k\}\right\},
 \qquad \delta(\boldsymbol\tau)=\dim_\Q S_{\boldsymbol\tau}.
\end{equation}
An empty product is $1$, so $S_{\varnothing}=\Q$. This is a finite-dimensional $\Q$-space; it need not be a field.

\begin{theorem}\label{thm:prescribedintro}
Fix $k\geq0$ and a tuple $\boldsymbol\tau\in\HH^k$. Work on a common simply connected smooth parameter disk, after finite base change, on which the conic and van Geemen cycle families on the resolved mirror quintic and their normal functions are defined. All but countably many fibers $Y_t$ have the following property. On
\[
 X=T_{\boldsymbol\tau}\times Y_t\times\PP^{p-1},\qquad p\geq1,
\]
the regulator images obtained from the conic and van Geemen cycle classes using every choice of $a$- and $b$-loops span a subspace of dimension \emph{exactly}
\begin{equation}\label{eq:exactintro}
 2\delta(\boldsymbol\tau)
\end{equation}
in $\Reg_{p,k}^{X}(L_pH_{2p+k}(X,\Q)_{\homol})$. All source classes are integral and have zero integral Betti class. A single fiber $Y_t$ may be selected simultaneously for any prescribed countable collection of such tuples, allowing varying $k$. In particular, one fiber works for every finite tuple of algebraic period parameters in $\HH$, as made precise in Corollary~\ref{cor:allalgebraic}.
\end{theorem}
The geometric cycle families and their differential equations are those of \cite{MW,LW,Walcher}. The calculation above combines them with separation from the \emph{whole} period group. It does not certify nonvanishing at a prescribed algebraic value of $t$.

\begin{corollary}\label{cor:CMintro}
Let $\ell_1,\ell_2,\ldots$ be distinct primes and put $\tau_j=i\sqrt{\ell_j}$. One can choose a single smooth mirror-quintic fiber $Y$ such that, for every $k\geq0$ and $p\geq1$,
\[
 E_{i\sqrt{\ell_1}}\times\cdots\times E_{i\sqrt{\ell_k}}
       \times Y\times\PP^{p-1}
\]
has a loop-generated detected subspace of dimension $2^{k+1}$. All elliptic factors have complex multiplication. For one repeated imaginary quadratic modulus $\tau$, the corresponding loop-generated subspace on $E_\tau^k\times Y\times\PP^{p-1}$ has dimension exactly $4$ for every $k\geq1$, after choosing $Y$ as above.
\end{corollary}
The different dimensions reflect relations among the loop periods: the distinct-prime choices give $2^k$ independent period monomials, whereas the powers of one quadratic number span a two-dimensional space. Neither assertion computes the full Lawson group or the full regulator image.

\begin{theorem}[Existence and dimension]\label{thm:allintro}
For every $p\geq1$ and $k\geq0$, there is a smooth projective variety of dimension $p+k+2$ carrying an integral class with zero Betti class and nonzero rational Lawson--Deligne image. More generally, such a class exists on some smooth projective $n$-fold precisely when $n\geq p+k+2$.
\end{theorem}
The broad existence and infinite-generation phenomena are already in \cite{HuAJ}. The present statement supplies a period-controlled realization and the sharp threshold for this particular invariant; it is not presented as a strengthening of those infinite-generation theorems.

\subsection{Geometric decompositions and rational examples}
For brevity, define the entire homologically trivial image by
\[
 \Img_{p,k}(X)=\im\bigl(\Reg_{p,k}^{X}|_{L_pH_{2p+k}(X,\Q)_{\homol}}\bigr).
\]
The symbol $\LoopIm_{p,k}$ denotes a specified subspace of $\Img_{p,k}$ generated by elliptic-loop constructions, not the full image. For $p=0$, Dold--Thom gives $\Img_{0,k}=0$. For $p<0$ we set $\Img_{p,k}=0$ in the splitting formulas, consistently with the usual extension of Lawson homology. Groups in cycle dimension greater than the dimension of the variety are also zero.

\begin{theorem}[Regulator-compatible blow-up decomposition]\label{thm:splitintro}
Let $V\subset X$ be a smooth center of codimension $r\geq2$, and let $B=\Bl_VX$. Under the standard blow-up correspondences,
\begin{equation}\label{eq:splitintro}
 \Img_{p,k}(B)\simeq\Img_{p,k}(X)
             \oplus\bigoplus_{j=1}^{r-1}\Img_{p-j,k}(V).
\end{equation}
The embeddings on the right carry the usual cohomological and Tate shifts, made explicit in Theorem~\ref{thm:fullblowup}. The projective-bundle decomposition is compatible in the same way.
\end{theorem}
The Lawson and Deligne decompositions themselves are established results \cite{FG,HuB,HuLi,HuLiFM}. Equation~\eqref{eq:splitintro} is their regulator-image consequence, not a new blow-up formula for Lawson homology.

\begin{theorem}[Independent exceptional images on rational varieties]\label{thm:rankintro}
For every $p\geq2$, $k\geq0$, and $s\geq1$, there is a smooth projective rational variety $B_s$ with
\[
 \dim B_s=p+2k+5,
 \qquad \dim_\Q\Img_{p,k}(B_s)\geq s\,2^{k+1}.
\]
The displayed lower bound is realized by an explicitly specified exceptional subspace of that exact dimension. It can be built using the complex multiplication factors in Corollary~\ref{cor:CMintro}.
\end{theorem}
\begin{remark}
The variety may change with $s$; no infinite-rank assertion on a fixed $B_s$ is inferred. For $p=2,k=1$ this gives rational ninefolds with a specified $4s$-dimensional detected subspace in $\HD^{13}(B_s,\Q(5))$. Existence of smooth rational varieties with very large Lawson kernels is already known from \cite{HuB}; the content here is the tracked Deligne image, its independent exceptional summands, and the prescribed-factor calculation.
\end{remark}

\subsection{Provenance and conventions}
The construction uses the standard cycle-space trace, current resolutions, and moving and intersection theorems cited below. The nontriviality and infinite-generation results of \cite{HuAJ,HuB}, the motivic formulas of \cite{HuLi,HuLiFM}, and the normal-function calculations of \cite{MW,LW,Walcher} are established results. We prove the separation and period-space calculations in the form required for the applications. Section~\ref{sec:chow} distinguishes the known Bloch--KLM comparison from the further realization comparison needed for the higher Chow composite.

A \emph{detected} class is an integral Lawson class with zero integral Betti class and nonzero rational Lawson--Deligne image. A rank of regulator images always means their $\Q$-linear dimension. In particular, an exact rank for $\LoopIm_{p,k}$ is not an exact rank for $\Img_{p,k}$. ``Very general'' here means outside a specified countable union of proper complex-analytic sets; it does not assert that the exceptional sets are algebraic.

There have been substantial improvements over the version of many years ago.

\section{Deligne cohomology, sparks, and filtered currents}\label{sec:cone}
\subsection{The integral and rational Deligne complexes}
For $A=\Z$ or $\Q$ and $q\geq1$, the Deligne complex is
\[
 A_{\mathcal D}(q)=\bigl[A(q)\longrightarrow\mathcal O_X
   \xrightarrow{d}\Omega_X^1\xrightarrow{d}\cdots
   \xrightarrow{d}\Omega_X^{q-1}\bigr],
 \qquad A(q)=(2\pi i)^q A,
\]
with $A(q)$ in degree zero. Its hypercohomology is $\HD^*(X,A(q))$. We use the standard current realization of this complex \cite{EV,HL,HLZ}.

\subsection{Normalization and the cone}
Write $\Dcur_X^j$ for currents of cohomological degree $j$ on the underlying real $2n$-manifold. Let $\Icur_X^\bullet$ be the usual locally integrally flat current complex, and put $\Icur_{X,\Q}^\bullet=\Icur_X^\bullet\otimes\Q$. Set
\[
 F^q\Dcur_X^j=\bigoplus_{a+b=j,\ a\geq q}\Dcur_X^{a,b}.
\]
All current degrees are \emph{cohomological}: an integration current $[C]$ over a real $d$-chain has degree $2n-d$, and $[C](\eta)=\int_C\eta$. For a degree-$j$ current we use the de Rham differential
\[
 (dT)(\eta)=(-1)^{j+1}T(d\eta),
 \qquad d[C]=(-1)^{j+1}[\partial C].
\]
Thus a geometric filling $\partial\widetilde c=c_f$, where $[c_f]$ has degree $M$, gives the cohomological primitive $B=(-1)^M[\widetilde c]$. This sign is used explicitly in Section~\ref{sec:AJ} and the product calculations. A current in $\Icur_{X,\Q}$ means a rational locally integrally flat current.

We use the normalized Tate convention: the rational current complex embeds into complex currents without the scalar $(2\pi i)^q$. Multiplication of all entries by $(2\pi i)^q$ identifies this convention with the standard Deligne complex containing $\Q(q)$. Thus period quotients below use ordinary rational periods. This normalization is particularly convenient for comparison of different filtration levels; it does not change any nonvanishing statement. Products, first Chern classes, and Gysin maps below are transported under these same identifications. In particular, their normalized current formulas use ordinary intersection and integration, with the degree and weight shifts stated in Section~\ref{sec:operations}.

For the integral version replace $\Icur_{X,\Q}$ below by $\Icur_X$. More uniformly, put $\Icur_{X,A}=\Icur_X\otimes A$ and use $\mathscr C_X(A,q)$ for this coefficient version. We retain $\mathscr C_X(q)=\mathscr C_X(\Q,q)$ when coefficients are suppressed. The rational model is
\begin{equation}\label{eq:cone}
 \mathscr C_X(q)=
 \Cone\!\left(\Icur_{X,\Q}^\bullet\oplus F^q\Dcur_X^\bullet
       \xrightarrow{\ (r,\omega)\mapsto r-\omega\ }\Dcur_X^\bullet\right)[-1].
\end{equation}
Its degree-$j$ elements are triples $(r,\omega,a)$ with
\[
 r\in\Icur_{X,\Q}^j,\quad \omega\in F^q\Dcur_X^j,
 \quad a\in\Dcur_X^{j-1}.
\]
With the standard shifted-cone convention, its differential is
\begin{equation}\label{eq:dcone}
 d_{\mathscr C}(r,\omega,a)=(dr,d\omega,\omega-r-da).
\end{equation}
The standard current resolutions identify the degree-$j$ cohomology of $\mathscr C_X(A,q)$ with $\HD^j(X,A(q))$ for $A=\Z$ or $\Q$. Equivalently, one may use the shifted cone of $\Icur_{X,A}^\bullet\to\Dcur_X^\bullet/F^q$. This is the cone description underlying the spark formulation \cite[Section~14]{HL}, \cite{EV}; the integral-current foundations are developed in \cite{HLZ}.

\begin{lemma}\label{lem:filteredcycle}
Let $A=\Z$ or $\Q$. If a closed current $R\in\Icur_{X,A}^j$ belongs to $F^q\Dcur_X^j$, then $(R,R,0)$ defines a class in $\HD^j(X,A(q))$. If $R_1-R_0=dS$ for $S\in\Icur_{X,A}^{j-1}\cap F^q\Dcur_X^{j-1}$, the two classes coincide.
\end{lemma}
\begin{proof}
The first assertion follows from~\eqref{eq:dcone}. For the second,
\[
 d_{\mathscr C}(S,S,0)=(R_1-R_0,R_1-R_0,0).
\]
\end{proof}

\subsection{The spark description}
In the Harvey--Lawson model, a spark of degree $M-1$ is represented by a pair $(a,R)$ with $dR=0$ and
\[
 d a=e-\pi_{<q}R,
\]
where $a$ is a current modulo $F^q$ and $e$ is a smooth projected form. The curvature map is $\delta_1[(a,R)]=e$ and the integral characteristic map is $\delta_2[(a,R)]=[R]$. The subgroup $\ker\delta_1$ is $\HD^M(X,\Z(q))$ \cite[Section~14]{HL}. Passing from the pair to the shifted cone uses the same equivalence relation. In particular, when $R\in F^q$ is closed, the pair $(0,R)$ represents the same class as $(R,R,0)$ in the cone.

The rational case is obtained by rationalizing the integral-current entry, not by forgetting it. The fact that the potential entry is zero never implies that the Deligne class is zero. Its integral-current entry and the allowed boundaries remain part of the class.

\subsection{The secondary subgroup}
Hodge--de Rham degeneration gives the exact sequence
\begin{equation}\label{eq:Deligneexact}
 0\longrightarrow J_q^{j-1}(X;\Q)
 \longrightarrow\HD^j(X,\Q(q))
 \xrightarrow{\delta_2}
 H^j(X,\Q)\cap F^qH^j(X,\C)
 \longrightarrow0.
\end{equation}
Here $\delta_2[(r,\omega,a)]=[r]$, and $J$ is as in~\eqref{eq:Jintro}. The inclusion of the secondary quotient sends a closed complex-current class $[a]$ to $[(0,0,a)]$; this fixes its sign. With integral coefficients the precise sequence has rightmost term
\[
 \ker\left\{H^j(X,\Z)\longrightarrow
                H^j(X,\C)/F^qH^j(X,\C)\right\},
\]
and leftmost term $H^{j-1}(X,\C)/(F^q+\im H^{j-1}(X,\Z))$ in the same normalization. In particular the integral Betti class is retained, including torsion; writing an intersection in complex cohomology would lose that torsion.

\begin{remark}\label{rem:nonhausdorff}
The word ``Jacobian'' in this context does not imply a complex torus. The projection of integral cohomology to $H^j/F^q$ need not be discrete; rational period groups are in any event countable subgroups, not complex-linear subspaces. We use these quotients as abelian groups, or as $\Q$-vector spaces with rational coefficients. No separation property is used.
\end{remark}

If $q\leq0$, then $F^q\Dcur_X^\bullet=\Dcur_X^\bullet$. The cone is quasi-isomorphic to $\Icur_{X,\Q}^\bullet$, and its secondary subgroup vanishes. Thus in this range the normalized construction contains only the Betti class.

\section{The Lawson--Deligne map without a metric}
\subsection{Trace currents and their type}
We recall precisely the trace input used here. Piecewise-linear approximation, group completion, and the incidence-cycle trace are part of the standard Lawson cycle-space construction \cite{F,FM,HuAJ}. They give a trace $c_f$ of real dimension $2p+k$ for an admissible representative $f:S^k\to\Zcyc_p(X)$, and a trace $c_F$ of dimension $2p+k+1$ for a homotopy, with
\[
 \partial c_f=0,\qquad \partial c_F=c_{f_1}-c_{f_0}.
\]
The approximation and trace results cited above allow representatives and homotopies to be taken in a finite-dimensional, stratified moving model after stabilization. Their traces are additive and compatible with subdivisions and prisms. We put
\[
 R_f=[c_f],\qquad R_F=(-1)^M[c_F].
\]
The convention in Section~\ref{sec:cone} then gives $dR_F=R_{f_1}-R_{f_0}$. The sign does not affect any Hodge-type estimate. In the effective-cycle description a family with base value $C$ is translated by the constant family $C$ to obtain a based map. For $k=0$, a based map from $S^0$ sends its nonbasepoint to the signed cycle in question.

\begin{proposition}[Trace type estimate]\label{prop:type}
For an admissible stratified piecewise-linear family of algebraic $p$-cycles with a real $r$-dimensional parameter chain, its trace current has only bidimensions
\[
 (p+a,p+b),\qquad a,b\geq0,\quad a+b=r.
\]
Equivalently, its cohomological Hodge components have types
\begin{equation}\label{eq:trace-type}
 (c-a,c-b),\qquad a,b\geq0,\quad a+b=r.
\end{equation}
In particular, for the indices~\eqref{eq:indices},
\begin{equation}\label{eq:two-estimates}
 R_f\in F^{q+1}\Dcur_X^M,
 \qquad R_F\in F^q\Dcur_X^{M-1}.
\end{equation}
\end{proposition}
\begin{proof}
On a smooth incidence stratum the tangent directions of the trace contain the complex $p$-dimensional tangent space of the algebraic fiber. A differential form of degree $2p+r$ can evaluate nontrivially only if its holomorphic and antiholomorphic degrees are both at least $p$. The remaining $r$ real directions account for the integers $a,b$. Integration over the incidence strata, and then addition over a triangulation, preserves this restriction; lower-dimensional strata do not change the current. Signed families are handled by linearity. Poincar\'e duality converts bidimension to cohomological type and gives~\eqref{eq:trace-type}. Its smallest possible first index is $c-r$. For $r=k$ this is $q+1$, and for $r=k+1$ it is $q$.
\end{proof}

\subsection{Descent and comparison with the Hodge spark}
\begin{proof}[Proof of the construction in Theorem~\ref{thm:mainmap}]
By~\eqref{eq:two-estimates}, $(R_f,R_f,0)$ is a cocycle in $\mathscr C_X(\Z,q)$. If $f_0$ and $f_1$ are homotopic, their trace difference is $dR_F$, with $R_F\in F^q$. The integral version of Lemma~\ref{lem:filteredcycle} therefore identifies the two classes. The argument also applies to the approximation and prism choices, and respects addition of based sphere families via the usual pinch map or the topological abelian group structure. For $k=0$, represent a cycle by a based map from $S^0$ sending its nonbasepoint to that signed cycle. Additivity of traces proves the homomorphism property. Changing coefficients in the cone and extending the homomorphism $\Q$-linearly gives $\Reg_{p,k}^{X}$ and its compatibility with $\ahat^X_{p,k}$. Finally, $[R_f]$ is by definition the Poincar\'e dual of the Friedlander--Mazur class. This proves~\eqref{eq:mainmap} and the compatibility with $\delta_2$. Section~\ref{sec:AJ} proves the assertion about the secondary subgroup, and Section~\ref{sec:operations} proves the stated compatibilities, with the rational correspondence statement in Theorem~\ref{thm:corr}.
\end{proof}

\begin{proposition}[The projected potential vanishes]\label{prop:potential}
Choose a K\"ahler metric, let $G$ be its Green operator, and let $\pi_{<q}$ project onto holomorphic current degrees below $q$. Then
\[
 -\pi_{<q}d^*G(R_f)=0.
\]
\end{proposition}
\begin{proof}
On a K\"ahler manifold the Laplacian and $G$ preserve bidegree. The operator $d^*=\partial^*+\bar\partial^*$ lowers either holomorphic or antiholomorphic degree by one. Every holomorphic degree in $R_f$ is at least $q+1$. Those in $d^*G(R_f)$ are therefore at least $q$, and $\pi_{<q}$ kills the result.
\end{proof}

Thus the Hodge spark $a_f:=-\pi_{<q}d^*G(R_f)$ has vanishing projected potential for a K\"ahler metric; see \cite[\S12]{HLZ} for the general construction. Its pair $(a_f,R_f)$ is therefore $(0,R_f)$, which agrees with the cone class. This identifies the filtered-current construction with its Hodge-spark description and proves independence of the chosen K\"ahler metric.

Vanishing of the projected potential does not imply that the Deligne class vanishes. If $R_f=dB$, triviality of $(0,R_f)$ requires that the class of the projected primitive $\bar B$ represent a period. Section~\ref{sec:AJ} makes this secondary condition explicit. The cone argument does not require pullback to commute with a Green operator.

\subsection{Why one filtration step is lost}
The trace $R_f$ already defines a class in $\HD^M(X,\Q(q+1))$ for a \emph{chosen representative}. A homotopy need only have its trace in $F^q$, however. At level $q+1$ its contribution is the class of $R_F$ modulo $F^{q+1}$ and periods, and it need not vanish. Lowering to $q$ is what gives homotopy invariance.

This issue is genuine even in degree zero. On an elliptic curve, a non-torsion point difference $[x]-[0]$ is zero in $L_0H_0$, but its ordinary Deligne divisor class in $\HD^2(E,\Q(1))$ is nonzero. Thus the full classical Abel--Jacobi invariant cannot in general descend through algebraic equivalence. At $k=0$ our map is a \emph{projection} of the classical invariant, not the full intermediate-Jacobian-valued map.

\section{Secondary periods and Hodge-theoretic limitations}\label{sec:AJ}
\subsection{An exact period formula}
\begin{proposition}[Secondary period formula]\label{prop:AJ}
Let $\alpha\in L_pH_{2p+k}(X,\Q)_{\homol}$ and let $R_f$ be a rational linear combination of trace currents representing $\alpha$. Choose $B\in\Icur_{X,\Q}^{M-1}$ with $dB=R_f$, and let $\bar B$ be its image in $\Dcur_X^\bullet/F^q$. Then $\bar B$ is closed and represents $\Reg_{p,k}^X(\alpha)$ in the secondary quotient, modulo rational periods. If $R_f=[c_f]$ and $\partial\widetilde c=c_f$, its functional description is
\begin{equation}\label{eq:AJformula}
 \omega\longmapsto (-1)^M\int_{\widetilde c}\omega,
 \qquad \partial\widetilde c=c_f,
\end{equation}
modulo the image of $H_{2p+k+1}(X,\Q)$, on
\begin{equation}\label{eq:testspace}
 V_{p,k}(X)=F^{p+k+2}H^{2p+k+1}(X,\C).
\end{equation}
The sign in~\eqref{eq:AJformula} converts a geometric filling to a cohomological-current primitive. It is $(-1)^k$, since $M=2c-k$.
\end{proposition}
\begin{proof}
In the quotient-cone model the differential is $(r,a)\mapsto(dr,-\bar r-da)$. Hence
\[
 (R_f,0)-d(B,0)=(0,\bar B).
\]
Because $dB=R_f\in F^q$, the projected current is closed. The current $B$ itself is not asserted to be closed. Hodge--de Rham degeneration identifies
\[
 H^{M-1}(\Dcur_X^\bullet/F^q)
 \simeq H^{M-1}(X,\C)/F^qH^{M-1}(X,\C).
\]
Since $R_f$ is exact, strictness supplies $S\in F^q\Dcur_X^{M-1}$ with $dS=R_f$. Thus $B-S$ is closed. A different such $S$ changes its class by $F^q$, whereas a different rational primitive $B$ changes it by rational cohomology. These are precisely the ambiguities in~\eqref{eq:Jintro}.

Poincar\'e duality identifies the annihilator of $F^qH^{M-1}(X,\C)$ with
\[
 F^{n-q+1}H^{2n-M+1}(X,\C)
   =F^{p+k+2}H^{2p+k+1}(X,\C).
\]
A smooth closed form in this filtration pairs trivially with $S$ by bidegree, so its pairing with $B-S$ is its pairing with $B$. For a geometric filling, $B=(-1)^M[\widetilde c]$, giving~\eqref{eq:AJformula}. A change of filling adds a rational cycle. To verify independence of a filtered de Rham representative, write the difference of two such representatives as $d\eta$ with $\eta$ in the same filtration, using strictness. Then the pairing with $B$ is, up to the fixed differential sign, the pairing of $R_f$ with $\eta$; it vanishes because $R_f\in F^{q+1}$. This also proves that the functional is defined on the cohomology space~\eqref{eq:testspace}.
\end{proof}

For $k=0$, the group $L_pH_{2p+k}(X,\Q)_{\homol}$ is the rational Griffiths group of $p$-cycles. The quotient of the usual Abel--Jacobi target by the additional $F^{c-1}$ directions kills algebraically trivial cycles. This is the precise meaning of the reduction at degree zero. For $k>0$, Proposition~\ref{prop:AJ} gives the reduced generalized Abel--Jacobi expression of \cite{HuAJ}, with the current-to-chain sign made explicit in~\eqref{eq:AJformula}.

\subsection{A vanishing criterion}
\begin{proposition}[Hodge gap required for detection]\label{prop:Hodgegap}
The secondary target is a quotient of
\[
 \bigoplus_{a+b=M-1,\ a<q}H^{a,b}(X).
\]
Every summand in this expression satisfies $b-a\geq k+3$. In particular, if
\begin{equation}\label{eq:Hodgevanish}
 H^{a,M-1-a}(X)=0\quad\text{for all }a<q,
\end{equation}
then $\Reg_{p,k}^X$ vanishes on $L_pH_{2p+k}(X,\Q)_{\homol}$.
\end{proposition}
\begin{proof}
Here $M-1=2q+k+1$. If $a\leq q-1$, then
\[
 b-a=M-1-2a\geq(2q+k+1)-2(q-1)=k+3.
\]
Condition~\eqref{eq:Hodgevanish} says $H^{M-1}=F^qH^{M-1}$, so the secondary quotient in~\eqref{eq:Deligneexact} is zero.
\end{proof}

\begin{corollary}\label{cor:Tate}
If the cohomology of $X$ is Hodge--Tate, the Lawson--Deligne map carries no information beyond the Betti class in any of the above bidegrees.
\end{corollary}
\begin{proof}
All nonzero Hodge components have $a=b$, whereas the components allowed by Proposition~\ref{prop:Hodgegap} have $b-a\geq3$.
\end{proof}
This is a limitation of this particular secondary invariant, not a proof that the corresponding Lawson kernel vanishes. In particular, vanishing of a regulator must not be confused with injectivity of a cycle map.

\begin{example}[The threefold limitation]\label{ex:threefold}
For $n=3,p=1,k=1$ the prescribed weight is $q=0$, and the construction reduces to
\[
 L_1H_3(X,\Q)\longrightarrow\HD^3(X,\Q(0))\simeq H^3(X,\Q).
\]
It is exactly $\PD\Phi_{1,3}$. Thus this particular invariant cannot distinguish a nonzero element in the kernel of $L_1H_3(X,\Q)\to H_3(X,\Q)$ on a threefold, even an abelian threefold. The fourfold example below is in the first dimension where this secondary detection is possible.
\end{example}

\begin{corollary}[Zero- and one-dimensional cycles]\label{cor:smallp}
For $p=0$ the secondary part of $\Reg_{0,k}^X$ vanishes for every $k$. For $p=1$ it is tested only on
\[
 V_{1,k}(X)=H^{k+3,0}(X).
\]
Consequently $H^{k+3,0}(X)=0$ implies that $\Reg_{1,k}^X$ vanishes on $L_1H_{k+2}(X,\Q)_{\homol}$.
\end{corollary}
\begin{proof}
For $p=0$, formula~\eqref{eq:testspace} gives $F^{k+2}H^{k+1}=0$. For $p=1$ it gives $F^{k+3}H^{k+3}=H^{k+3,0}$. Proposition~\ref{prop:AJ} proves both assertions; when $q\leq0$ they also follow directly from the cone.
\end{proof}

\begin{corollary}[Rational one-cycle vanishing, cf.~\cite{HuB}]\label{cor:rationalone}
For a smooth projective rational variety $X$, $\Img_{1,k}(X)=0$ for every $k\geq0$.
\end{corollary}
\begin{proof}
A smooth projective rational variety has $H^{j,0}=0$ for $j>0$, by birational invariance of holomorphic forms. Apply Corollary~\ref{cor:smallp}. Alternatively, the stronger assertion $L_1H_{k+2}(X)_{\homol}=0$ is already contained in the birational invariance theorem of \cite{HuB}.
\end{proof}

\section{Functoriality, algebraic correspondences, and geometric operations}\label{sec:operations}
We record the geometric compatibilities needed in the applications. Their proofs use traces and filtered currents, rather than commutation of pullback with a Green operator.

\begin{proposition}\label{prop:operations}
The following operations commute with $\Reg$ and with its integral version $\ahat$, with the indicated indices and the corresponding coefficient groups.
\begin{enumerate}[label=(\roman*),leftmargin=2em]
\item For a proper morphism $h:X^n\to Y^d$, pushforward on $L_pH_{2p+k}$ corresponds to Deligne pushforward, of cohomological degree $2(d-n)$ and weight shift $d-n$.
\item For a smooth projective morphism $g:X^{n+e}\to Y^n$ of relative dimension $e$, pullback from $L_pH_{2p+k}(Y)$ to $L_{p+e}H_{2p+2e+k}(X)$ corresponds to Deligne pullback. Both $M$ and $q$ are unchanged.
\item For a line bundle $L$ on $X$ and $p\geq1$,
\begin{equation}\label{eq:c1compat}
 \Reg_{p-1,k}^X\big(c_1(L)\cap\alpha\big)
 =c_1^{\mathcal D}(L)\smile\Reg_{p,k}^X(\alpha).
\end{equation}
The right side uses the \emph{ordinary} Deligne first Chern class in weight $1$. Its target is $\HD^{M+2}(X,\Q(q+1))$.
\end{enumerate}
The statements use the cone convention of Section~\ref{sec:cone} when a source or target weight is nonpositive. Cycle dimensions must be nonnegative; a group in cycle dimension exceeding that of its variety is zero.
\end{proposition}
\begin{proof}
A proper holomorphic map induces pushforward on currents, commutes with their differential, and changes Hodge type by $(d-n,d-n)$. Trace commutes with proper pushforward. Thus it sends the cocycle $(R_f,R_f,0)$ to the cocycle of the pushed-forward family. This proves (i). For (ii), smooth pullback is defined on currents, preserves the Hodge filtration, and commutes with the incidence trace and its prism construction.

For (iii) use the usual moving model for the Lawson intersection operation \cite{FG}. In that model proper intersection with an algebraic Cartier divisor commutes with incidence trace. On properly intersecting representatives, multiplication of the Deligne cycle cocycles is represented by the intersection of their currents; equivalently this follows from the current product model of Deligne cohomology \cite{HL} and its mixed-level analytic formula \cite[Section~6]{Hao}. Thus~\eqref{eq:c1compat} holds on such representatives. The moving homotopies on the Lawson side give trace primitives with first Hodge index at least $q+1$, the level appropriate to the resulting $(p-1,k)$ group. Lemma~\ref{lem:filteredcycle} makes these changes invisible in the target. This proves independence of the moving representative. Differences of Cartier divisors defining the same line bundle have the same ordinary Deligne cycle class, so the formula depends only on $L$.
\end{proof}

\begin{remark}\label{rem:product}
For an external product of two Lawson classes, the weight prescribed by~\eqref{eq:indices} is $q_1+q_2+1$, not $q_1+q_2$. Therefore the collection $\Reg_{p,k}$ is not asserted to be a ring map into the usual bigraded Deligne ring. Our product examples are computed directly from traces. The Chern-class action in~\eqref{eq:c1compat} is different: it uses an ordinary algebraic Deligne class with its usual weight.
\end{remark}

\subsection{Algebraic cycle actions and correspondences}
We use the intersection and moving constructions of Friedlander--Gabber \cite{FG}, and the correspondence formalism of Hu--Li \cite[Section~4]{HuLi}. The following statement is the compatibility of the present regulator with those already established operations. An intersection of currents below is used only in the good-position model for the corresponding algebraic cycle operation; arbitrary currents are not being multiplied. The mixed-level Deligne product agrees with the usual product, as in \cite[Theorems~6.6 and~6.11]{Hao}; its algebraic-cycle specialization is \cite[Propositions~7.1--7.3]{Hao}.

\begin{lemma}[Action of an ordinary algebraic cycle]\label{lem:cycleaction}
Let $0\leq a\leq p$, let $Z\in\CH^a(X)$, and let $\alpha\in L_pH_{2p+k}(X)$. Here $\CH^a(X)$ denotes the Chow group of codimension-$a$ cycles. Then
\[
 \ahat^X_{p-a,k}(Z\cap\alpha)=\clD(Z)\smile\ahat^X_{p,k}(\alpha)
      \quad\text{in }\HD^{M+2a}(X,\Z(q+a)).
\]
The same equality holds with rational coefficients. The Deligne class $\clD(Z)$ here has its ordinary weight $a$, not the shifted weight of a degree-zero Lawson regulator.
\end{lemma}
\begin{proof}
In the good-position model for the intersection operation, the trace of $Z\cap f$ is $[Z]\cap R_f$, and the trace of an intersected prism is $[Z]\cap R_F$. Intersecting with the complex codimension-$a$ cycle adds $(a,a)$ to cohomological type. Hence the cycle trace is in $F^{q+a+1}$ and the prism trace is in $F^{q+a}$. On a common good-position model, a representative of the usual cone product of closed triples of degrees $u$ and $v$ is
\[
 (r_1,\omega_1,a_1)\cdot(r_2,\omega_2,a_2)
 =\bigl(r_1\cap r_2,\ \omega_1\cap\omega_2,\
 a_1\cap\omega_2+(-1)^u r_1\cap a_2\bigr).
\]
The differential~\eqref{eq:dcone} verifies the product identity; the cited Deligne product theorem identifies it with the ordinary product. Only the intersections admitted by the moving model are used here. In particular, the ordinary cycle cocycle $([Z],[Z],0)$ and the trace cocycle $(R_f,R_f,0)$ have product
\[
 ([Z]\cap R_f,[Z]\cap R_f,0).
\]
This is the cocycle for the Lawson intersection. Changes made by the moving construction are homotopies of the resulting $(p-a)$-cycle families; their primitives belong to $F^{q+a}$ and are cone boundaries. Thus the identity is independent of good-position choices. The usual algebraic intersection operation and $\clD$ both factor through rational equivalence, so the identity depends only on $Z\in\CH^a(X)$. This uses the established moving and current-product theorems, rather than an extension of a wedge product to arbitrary currents.
\end{proof}

\begin{theorem}[Correspondence compatibility]\label{thm:corr}
Let $X^n,Y^m$ be smooth projective varieties, let $t\in\Z$, and let $\Gamma\in\CH^{n+t}(X\times Y)_\Q$. Put $p'=p+m-n-t$ and assume $p'\geq0$. The usual correspondence action has the form
\[
 \Gamma_*:L_pH_{2p+k}(X,\Q)
       \longrightarrow L_{p'}H_{2p'+k}(Y,\Q),
 \quad \Gamma_*\alpha=(\operatorname{pr}_Y)_*
           (\Gamma\cap\operatorname{pr}_X^*\alpha).
\]
It satisfies
\begin{equation}\label{eq:corrcompat}
 \Reg_{p',k}^{Y}(\Gamma_*\alpha)
   =(\operatorname{pr}_Y)_*
       \bigl(\clD(\Gamma)\smile\operatorname{pr}_X^*
                    \Reg_{p,k}^{X}(\alpha)\bigr).
\end{equation}
The indices on $Y$ are $M_Y=M+2t$ and $q_Y=q+t$, so the right side lies in $\HD^{M+2t}(Y,\Q(q+t))$. Nonpositive weights use the stated cone convention. Composition respects this equality whenever the intermediate cycle dimensions are defined. Integral correspondences give the same assertion for integral input and the integral map whenever the displayed cycle dimensions are nonnegative.
\end{theorem}
\begin{proof}
Pullback along $\operatorname{pr}_X$ is smooth of relative dimension $m$. It raises the cycle dimension to $p+m$ without changing $M$ or $q$. Intersection with the codimension-$(n+t)$ cycle $\Gamma$ gives dimension $p'$, increases cohomological degree by $2(n+t)$ and weight by $n+t$. Pushforward along $\operatorname{pr}_Y$ subtracts $2n$ from cohomological degree and $n$ from weight. Proposition~\ref{prop:operations} and Lemma~\ref{lem:cycleaction} therefore give exactly~\eqref{eq:corrcompat}. The condition $p'\geq0$ is precisely the condition needed for that intersection action. The equality for composition follows from the established functoriality of correspondence actions on Lawson homology \cite{HuLi}, and from the projection and base-change formulas for ordinary Deligne cycle classes. These operations preserve the Betti realization, so they also preserve homological triviality.
\end{proof}

\subsection{Recovery of the morphic functoriality statement}
By Friedlander--Lawson duality \cite{FL},
\[
 D_X:L^cH^{2c-k}(X)\xrightarrow{\sim}L_{n-c}H_{2(n-c)+k}(X).
\]
Keeping the earlier notation, set
\[
 \bar a_X=\ahat^X_{n-c,k}\circ D_X:
 L^cH^{2c-k}(X)\longrightarrow
       \HD^{2c-k}(X,\Z(c-k-1)).
\]
\begin{corollary}\label{cor:morphic}
For a projective morphism $f:X\to Y$ between smooth projective varieties, the maps $\bar a$ commute with $f^*$ whenever the corresponding Lawson cycle dimensions on $X$ and $Y$ are nonnegative. The rational versions commute as well. In particular, this applies for $0\leq c\leq\min(\dim X,\dim Y)$ and $c-k-1\geq1$.
\end{corollary}
\begin{proof}
The transpose graph of $f$ gives the Lawson Gysin pullback corresponding to morphic pullback under $D$. Apply Theorem~\ref{thm:corr} to that graph; its cohomological shift is zero. The right side is ordinary Deligne pullback. This proves morphic functoriality without pulling back arbitrary currents or commuting pullback with a Green operator.
\end{proof}

\begin{corollary}[Projectors]\label{cor:projectors}
If $\pi\in\CH^n(X\times X)_\Q$ is idempotent, the regulator restricts to its image on Lawson homology and takes values in its Deligne image. The same holds on the homologically trivial subgroup.
\end{corollary}
\begin{proof}
Use~\eqref{eq:corrcompat} with $t=0$ and $\pi^2=\pi$. This is a realization-compatible consequence of the correspondence formalism of \cite{HuLi}.
\end{proof}

\section{Picard--Fuchs equations as nonvanishing certificates}\label{sec:PF}
\subsection{The general criterion}
Let $\mathcal Y\to U$ be a smooth projective family of threefolds over a simply connected complex disk. Assume $h^{3,0}(Y_t)=1$ and choose a nowhere-vanishing holomorphic section $\Omega_t$ of the resulting Hodge line bundle. Let $Z_t$ be a relative algebraic $1$-cycle with zero integral homology class on every fiber. Choose a continuous local family of bounding chains $U_t$, and write
\[
 T(t)=\int_{U_t}\Omega_t,\qquad \partial U_t=Z_t.
\]
We assume that $T$ is a holomorphic branch of the resulting truncated normal function. Fix a flat integral basis of $H_3(Y_t,\Z)$ modulo torsion on $U$, and denote its holomorphic period functions by $\Pi_1,\ldots,\Pi_b$. Any change of the chosen chain branch adds a fixed integral combination of these periods.

\begin{theorem}[Picard--Fuchs detection]\label{thm:PF}
Suppose a linear differential operator $\mathcal L$, with holomorphic coefficients on $U$, satisfies
\begin{equation}\label{eq:PFhyp}
 \mathcal L\Pi_j=0\quad(1\leq j\leq b),
 \qquad \mathcal LT=g\not\equiv0.
\end{equation}
Then outside a countable subset of $U$,
\[
 T(t)\notin\sum_{j=1}^b\Q\Pi_j(t).
\]
Consequently $[Z_t]\in L_1H_2(Y_t)$ has zero integral Betti class and a nonzero rational Lawson--Deligne image in $\HD^4(Y_t,\Q(1))$.
\end{theorem}
\begin{proof}
For $a=(a_1,\ldots,a_b)\in\Q^b$, put
\[
 F_a(t)=T(t)-\sum_j a_j\Pi_j(t).
\]
The equation $\mathcal LF_a=g\not\equiv0$ shows that $F_a$ is not identically zero. Its zero set is discrete and hence countable. The set $\Q^b$ is countable, so the union of all these zero sets is countable. At every remaining point, $T(t)$ is not a rational period.

For $n=3,p=1,k=0$, one has $M=4$ and $q=1$. Moreover,
\[
 J_1^3(Y_t;\Q)
 \simeq H^{3,0}(Y_t)^*/\im H_3(Y_t,\Q).
\]
Proposition~\ref{prop:AJ} identifies the regulator of $Z_t$ with the class of $T(t)$ in this quotient. The assertion follows.
\end{proof}

This proof excludes the \emph{full} rational period group, not only a chosen fundamental period. It requires neither numerical precision nor an unproved transcendence statement at a specified algebraic parameter. It also does not establish nonvanishing on \emph{every} fiber.

\begin{theorem}[Independence from inhomogeneous terms]\label{thm:PFrank}
In the same family, let $Z_{1,t},\ldots,Z_{d,t}$ be homologically trivial integral algebraic $1$-cycles with holomorphic truncated normal functions $T_1,\ldots,T_d$. Suppose
\[
 \mathcal L\Pi_j=0\quad(1\leq j\leq b),\qquad
 \mathcal LT_\ell=g_\ell\quad(1\leq\ell\leq d),
\]
and that $g_1,\ldots,g_d$ are linearly independent over the constant field $\Q$ as holomorphic functions on $U$. Then, outside a countable subset of $U$, the $d$ classes
\[
 \Reg_{1,0}^{Y_t}([Z_{\ell,t}]\otimes1)
 \quad(1\leq\ell\leq d)
\]
are $\Q$-linearly independent.
\end{theorem}
\begin{proof}
For $a\in\Q^d\setminus\{0\}$ and $b'\in\Q^b$, set
\[
 F_{a,b'}(t)=\sum_{\ell=1}^d a_\ell T_\ell(t)
              -\sum_{j=1}^b b'_j\Pi_j(t).
\]
Its image under $\mathcal L$ is $\sum_\ell a_\ell g_\ell\not\equiv0$, so $F_{a,b'}$ is not identically zero and has a discrete zero set. There are countably many pairs $(a,b')$. Away from their combined zero sets, no nonzero rational combination of the $T_\ell(t)$ is a rational period. The identification used in Theorem~\ref{thm:PF} converts this statement into the asserted independence.
\end{proof}

\begin{theorem}[Separation over a prescribed countable coefficient space]\label{thm:coefficientPF}
In the setting of Theorem~\ref{thm:PFrank}, assume the inhomogeneous functions $g_1,\ldots,g_d$ are linearly independent over the constant field $\C$. Let $S\subset\C$ be a fixed countable $\Q$-vector space containing $1$. Write
\[
 \Lambda_\Q(t)=\sum_{j=1}^b\Q\Pi_j(t),\qquad
 S\Lambda_\Q(t)=\left\{\sum_{j=1}^b s_j\Pi_j(t):s_j\in S\right\}.
\]
For all but countably many $t\in U$, the $\Q$-linear map
\begin{equation}\label{eq:coeffsep}
 S^{\oplus d}\longrightarrow\C/S\Lambda_\Q(t),
 \qquad(s_1,\ldots,s_d)\longmapsto
          \left[\sum_{\ell=1}^d s_\ell T_\ell(t)\right]
\end{equation}
is injective. The exceptional set can be chosen simultaneously for any countable collection of coefficient spaces $S$. Changing a bounding-chain branch adds a period to $T_\ell$ and hence leaves the map in~\eqref{eq:coeffsep} unchanged.
\end{theorem}
\begin{proof}
For a nonzero $(s_\ell)\in S^d$ and any $(b_j)\in S^b$, the function
\[
 F_{s,b}(t)=\sum_\ell s_\ell T_\ell(t)-\sum_j b_j\Pi_j(t)
\]
is not identically zero: its image under $\mathcal L$ is $\sum_\ell s_\ell g_\ell$, which is nonzero by independence over $\C$. Its zero set is discrete. There are only countably many such pairs $(s,b)$, and hence only countably many excluded points. The same argument followed by another countable union treats countably many spaces. No claim of $S$-linearity is made unless $S$ is a field; the assertion needed here is $\Q$-linearity.
\end{proof}

The order of choices is therefore reversed: one may prescribe the elliptic period coefficients and then choose the fiber $Y_t$, rather than fix a reduced Abel--Jacobi value and vary the elliptic moduli. This argument requires independence over $\C$, not merely over $\Q$.

\subsection{An explicit mirror-quintic example}
On the Dwork pencil
\[
 Q_\psi=\{x_1^5+\cdots+x_5^5-5\psi x_1x_2x_3x_4x_5=0\}\subset\PP^4
\]
consider the conics
\begin{equation}\label{eq:conics}
 C_\pm=\{x_1+x_2=0,\ x_3+x_4=0,
                 \ x_5^2\pm\sqrt{5\psi}\,x_1x_3=0\}.
\end{equation}
Morrison--Walcher construct the corresponding homologically trivial difference $Z_\psi$ on the \emph{smooth resolved mirror quintic} $Y_\psi$, not on the unresolved quotient, and compute its truncated normal function \cite[Sections 3.6 and 4]{MW}. With their normalization, $z=(5\psi)^{-5}$ and
\begin{equation}\label{eq:MW}
 \begin{gathered}
 \mathcal L=\theta^4-5z(5\theta+1)(5\theta+2)(5\theta+3)(5\theta+4),
 \qquad\theta=z\frac{d}{dz},\\
 \mathcal LT=\frac{15}{16\pi^2}\sqrt z.
 \end{gathered}
\end{equation}
All periods of $\Omega_z$ are annihilated by $\mathcal L$. Work on a disk of smooth fibers, with branches chosen after the necessary finite base change and away from $z=0$. The last term of~\eqref{eq:MW} is not identically zero, so Theorem~\ref{thm:PF} applies.

\begin{corollary}\label{cor:initialdata}
There exists a smooth projective threefold $Y$, an integral homologically trivial algebraic $1$-cycle $Z$, a chain $U$ with $\partial U=Z$, and a nonzero $\Omega\in H^{3,0}(Y)$ such that
\begin{equation}\label{eq:initialperiod}
 A:=\int_U\Omega\notin\Lambda_\Q,
 \qquad
 \Lambda_\Q:=\left\{\int_\gamma\Omega:\gamma\in H_3(Y,\Q)\right\}.
\end{equation}
In fact the mirror-quintic construction gives this on all but countably many points of every sufficiently small smooth parameter disk.
\end{corollary}

\subsection{Two independent classes from inhomogeneous terms}
The van Geemen cycle difference supplies a second normal function on the same resolved mirror family. Laporte--Walcher \cite[Section~2, equation~(7)]{LW} give, for $\omega^2+\omega+1=0$,
\begin{equation}\label{eq:vG}
 \mathcal LT_{\mathrm{vG}}
 =\frac{1+2\omega}{(2\pi i)^2}\frac{32}{45}
 \frac{63\psi^{-5}+1824\psi^{-10}-512\psi^{-15}}
      {(1-128/(3\psi^5))^{5/2}}.
\end{equation}
Here $T_{\mathrm{vG}}$ is the chain integral of the difference of the two conjugate algebraic cycle families, with an orientation fixed as in that formula. Both this equation and~\eqref{eq:MW} use the standard Picard--Fuchs gauge for $\Omega_z$. Their occurrence with a common choice of the period operator is also recorded in \cite[Section~2.2, equations~(2.10)--(2.12)]{Walcher}. When comparing their displayed constants, note that $(2\pi i)^{-2}=-(4\pi^2)^{-1}$ and that reversing the orientation of a cycle difference reverses its normal function. We use exactly the normalization in~\eqref{eq:vG}. Also, \cite{Walcher} assigns $\pm15\sqrt z/(32\pi^2)$ to each conic separately, so their difference has the factor $15/(16\pi^2)$ in~\eqref{eq:MW}. Signs or fixed nonzero factors arising from the choice of cycle difference do not affect the independence argument; an arbitrary nonconstant rescaling of $\Omega_z$ is not being made.

\begin{corollary}[Two independent reduced Abel--Jacobi classes]\label{cor:twoclasses}
The conic and van Geemen inhomogeneous terms are linearly independent over $\C$ on a common branch. Consequently, the example in Corollary~\ref{cor:initialdata} can be chosen together with an integral homologically trivial algebraic $1$-cycle $Z_2$ and a chain $U_2$ with $\partial U_2=Z_2$ such that, on this same $Y$ and for this same $\Omega$, the classes of
\[
 A_1=A=\int_U\Omega,\qquad A_2=\int_{U_2}\Omega
\]
are $\Q$-linearly independent in $\C/\Lambda_\Q$. Thus the reduced Lawson--Deligne image of this fixed $Y$ has dimension at least two. The assertion holds outside a countable subset of a common smooth parameter disk.
\end{corollary}
\begin{proof}
Make a common finite base change so that the conic and van Geemen cycle families are simultaneously defined, and choose local holomorphic branches of their chain integrals. Let $T_1=T$ and $T_2=T_{\mathrm{vG}}$, with inhomogeneous terms $g_1,g_2$ given by~\eqref{eq:MW} and~\eqref{eq:vG}. On a punctured sector near $z=0$, the equality $\psi^{-5}=5^5z$ gives
\[
 g_1(z)=a z^{1/2},\qquad
 g_2(z)=b z+O(z^2),\qquad a,b\in\C^*.
\]
The factors $a,b$ are nonzero because $1+2\omega\ne0$ and the leading coefficient $63$ in~\eqref{eq:vG} is nonzero. Comparing the two different leading exponents shows that $g_1,g_2$ are linearly independent even over the constant field $\C$. The same conclusion holds on any connected common branch reached by analytic continuation, and hence on a sufficiently small smooth disk there. Apply Theorem~\ref{thm:PFrank} with $d=2$ and fix a fiber outside its countable exceptional set. Taking $Z_1=Z$ and $U_1=U$ preserves the notation of~\eqref{eq:initialperiod} and proves the assertion.
\end{proof}

This application uses two published geometric calculations to prove independence modulo the complete rational period group. It does not require a transcendence calculation at a specified algebraic parameter, and it does not by itself establish an infinite-dimensional image on the fixed fiber.

\begin{remark}[An exact local check]\label{rem:seriescheck}
With $w=\psi^{-5}$, the rational-algebraic factor of the van Geemen inhomogeneity expands as
\begin{equation}\label{eq:seriescheck}
 \frac{63w+1824w^2-512w^3}{(1-128w/3)^{5/2}}
   =63w+8544w^2+695808w^3+O(w^4).
\end{equation}
Since $w=5^5z$, its first exponent in $z$ is $1$, while the conic difference starts at exponent $1/2$. This calculation checks the nonzero leading coefficients used in Corollary~\ref{cor:twoclasses}; it does not recompute either geometric normal function. The supporting exact-arithmetic script also checks the degree and weight identities used later.
\end{remark}

\section{Elliptic products and the complete period denominator}\label{sec:elliptic}
Fix the geometric data $(Y,Z,U,\Omega)$ provided by Corollary~\ref{cor:initialdata}. The set $\Lambda_\Q\subset\C$ is countable, and $A\notin\Lambda_\Q$.

\subsection{One elliptic factor}
Recall from Section~\ref{subsec:ellipticparameters} that $\HH$ is the complex upper half-plane. For $\tau\in\HH$, write
\[
 E_\tau=\C/(\Z+\tau\Z)
\]
and choose the marked $a$-loop $\gamma:S^1\to E_\tau$ with period $\int_\gamma du=1$. The based family
\begin{equation}\label{eq:loop}
 f(t)=([\gamma(t)]-[0])\times Z
       \quad\text{in }\Zcyc_1(E_\tau\times Y)
\end{equation}
defines $\alpha_\tau\in L_1H_3(E_\tau\times Y)$.

Its trace is $\gamma\times Z$. With the convention
$\partial(a\times b)=\partial a\times b+(-1)^{\dim a}a\times\partial b$,
a geometric bounding chain is $-\gamma\times U$. Here $M=5$, so its cohomological-current primitive is $[\gamma\times U]$. Consequently the integral Betti class vanishes and the regulator value on $du\wedge\Omega$ is $A$, by~\eqref{eq:AJformula}.

The K\"unneth decomposition is essential. Among the summands of
$H_4(E_\tau\times Y,\Q)$, only
$H_1(E_\tau,\Q)\otimes H_3(Y,\Q)$ pairs with $du\wedge\Omega$. Therefore its \emph{entire} rational period group is
\begin{equation}\label{eq:ellipticperiod}
 \Lambda_\Q+\tau\Lambda_\Q.
\end{equation}
It is not enough to know that $A\notin\Lambda_\Q$.

\begin{lemma}[Explicit exceptional set]\label{lem:badtaus}
The set
\begin{equation}\label{eq:badtaus}
 \mathcal E_A=
 \left\{\frac{A-\lambda}{\mu}:\lambda,\mu\in\Lambda_\Q,
                                     \ \mu\ne0\right\}\cap\HH
\end{equation}
is countable. For $\tau\notin\mathcal E_A$,
$A\notin\Lambda_\Q+\tau\Lambda_\Q$.
\end{lemma}
\begin{proof}
If $A=\lambda+\tau\mu$, then $\mu\ne0$ because $A\notin\Lambda_\Q$, and $\tau=(A-\lambda)/\mu$. The set of pairs $(\lambda,\mu)$ is countable.
\end{proof}

\begin{theorem}[A detected fourfold class]\label{thm:fourfold}
For every $\tau\in\HH\setminus\mathcal E_A$, the family~\eqref{eq:loop} satisfies
\[
 \Phi_{1,3}(\alpha_\tau)=0,
 \qquad
 0\ne\Reg_{1,1}^{E_\tau\times Y}(\alpha_\tau\otimes1)
       \in\HD^5(E_\tau\times Y,\Q(1)).
\]
In particular $\alpha_\tau$ has infinite order and is not null-homotopic in the cycle space.
\end{theorem}
\begin{proof}
For this fourfold, $M=5$, $q=1$, and the secondary target is
\[
 J_1^4(E_\tau\times Y;\Q)
 \simeq\C/(\Lambda_\Q+\tau\Lambda_\Q),
\]
using the nonzero top form $du\wedge\Omega$. The class is $[A]$. Lemma~\ref{lem:badtaus} proves nonvanishing. A torsion Lawson class has zero image after tensoring with $\Q$, so this integral class has infinite order.
\end{proof}

\subsection{Arbitrary homotopy degree}
For $k\geq1$, choose $\boldsymbol\tau=(\tau_1,\ldots,\tau_k)\in\HH^k$ and set
\[
 T_{\boldsymbol\tau}=E_{\tau_1}\times\cdots\times E_{\tau_k},
 \qquad X_k=T_{\boldsymbol\tau}\times Y.
\]
For each factor let $\gamma_j$ be its marked $a$-loop. The family
\begin{equation}\label{eq:smashfamily}
 (t_1,\ldots,t_k)\longmapsto
 \prod_{j=1}^k([\gamma_j(t_j)]-[0])\times Z
\end{equation}
vanishes whenever one coordinate is the basepoint. The external products here are products of signed algebraic $0$-cycles. Thus~\eqref{eq:smashfamily} factors through
\[
 S^1\wedge\cdots\wedge S^1\simeq S^k
\]
and defines $\alpha_{\boldsymbol\tau}\in L_1H_{k+2}(X_k)$.

\begin{lemma}\label{lem:kperiods}
The trace of~\eqref{eq:smashfamily} is
$\gamma_1\times\cdots\times\gamma_k\times Z$. It is bounded by
$(-1)^k\gamma_1\times\cdots\times\gamma_k\times U$. On the top form
\[
 \Omega_k=du_1\wedge\cdots\wedge du_k\wedge\Omega
\]
the geometric bounding-chain integral is $(-1)^kA$. Since $M=k+4$, the corresponding cohomological-current primitive is $[\gamma_1\times\cdots\times\gamma_k\times U]$, and the regulator value in~\eqref{eq:AJformula} is $A$. Its complete rational period group is
\begin{equation}\label{eq:periodk}
 \Lambda_{\boldsymbol\tau,\Q}
 =\sum_{I\subseteq\{1,\ldots,k\}}
       \left(\prod_{i\in I}\tau_i\right)\Lambda_\Q.
\end{equation}
\end{lemma}
\begin{proof}
Expand the signed product in~\eqref{eq:smashfamily}. Every term containing a fixed basepoint is independent of at least one real parameter and has zero top-dimensional trace. The only surviving term is the stated product. The boundary sign follows from the product boundary rule.

For the period assertion, the only K\"unneth summand of $H_{k+3}(X_k,\Q)$ pairing with $\Omega_k$ is
\[
 H_1(E_{\tau_1},\Q)\otimes\cdots\otimes H_1(E_{\tau_k},\Q)
                       \otimes H_3(Y,\Q).
\]
The two marked periods in each elliptic factor are $1,\tau_j$, which gives precisely~\eqref{eq:periodk}.
\end{proof}

\begin{proposition}[Generic avoidance of all periods]\label{prop:avoidall}
Outside a countable union of proper analytic hypersurfaces in $\HH^k$,
\[
 A\notin\Lambda_{\boldsymbol\tau,\Q}.
\]
The complement is dense and nonempty.
\end{proposition}
\begin{proof}
For each tuple $(\lambda_I)_I$ of elements of the countable group $\Lambda_\Q$, consider the polynomial
\[
 P_{(\lambda_I)}(\boldsymbol\tau)
 =\sum_I\lambda_I\prod_{i\in I}\tau_i-A.
\]
It cannot be identically zero: independence of the distinct monomials would force $\lambda_\varnothing=A$, contrary to~\eqref{eq:initialperiod}. Its zero set is therefore a proper analytic hypersurface or is empty. There are only countably many tuples. Each zero set is closed and nowhere dense in $\HH^k$, and the Baire category theorem gives the final assertion.
\end{proof}

\begin{proof}[Construction at dimension $p+k+2$ in Theorem~\ref{thm:allintro}]
First take $p=1$. For $k=0$ the assertion is Corollary~\ref{cor:initialdata}; for $k\geq1$ use the smash-product family. The variety $X_k$ has dimension $k+3$, so $M=k+4$ and $q=1$. The secondary target is
\[
 J_1^{k+3}(X_k;\Q)
 \simeq H^{k+3,0}(X_k)^*/\im H_{k+3}(X_k,\Q)
 \simeq\C/\Lambda_{\boldsymbol\tau,\Q}.
\]
Lemma~\ref{lem:kperiods} and Proposition~\ref{prop:avoidall} show that the regulator is the nonzero class of $A$.

For general $p\geq1$, let $X=X_k\times\PP^{p-1}$ and pull the class back along the projection. Equivalently multiply the family by the whole algebraic cycle $[\PP^{p-1}]$. This increases the cycle dimension from $1$ to $p$ and the homological degree from $k+2$ to $2p+k$, while keeping $M=k+4$ and $q=1$. Smooth-pullback compatibility gives the pulled-back Deligne class. Pullback in Deligne cohomology is injective because a section of the projection is a left inverse. The resulting class is therefore detected. Its Betti class remains zero, and $\dim X=p+k+2$.
\end{proof}

\begin{corollary}[Sharp dimension threshold]\label{cor:dimension}
Fix $p\geq1$ and $k\geq0$. A class with zero Betti realization and nonzero rational realization under the construction of this paper exists on some smooth projective $n$-fold if and only if
\[
 n\geq p+k+2.
\]
\end{corollary}
\begin{proof}
If $n\leq p+k+1$, then $q=n-p-k-1\leq0$, and the cone has no secondary subgroup. This proves necessity. The construction immediately preceding this corollary proves sufficiency at $n_0=p+k+2$. For $n=n_0+s$, take the product with $\PP^s$ and push the detected class forward along $i:X\hookrightarrow X\times\PP^s$ at a fixed point. Its cycle dimension and homotopy degree do not change. Proper-pushforward compatibility identifies its regulator with $i_*$ of the original regulator, which is nonzero because $\operatorname{pr}_{X*}i_*$ is the identity in Deligne cohomology. The Betti class is still zero.
\end{proof}
The threshold concerns this regulator, not the possible existence of undetected Lawson kernels in smaller dimensions. The restriction $p\geq1$ is essential to the existence statement.

\subsection{Rank preservation and amplification}
\begin{proposition}[A simultaneous period-avoidance map]\label{prop:rank}
Let $W$ be a countable-dimensional $\Q$-subspace of $\C/\Lambda_\Q$. For $\boldsymbol\tau$ outside a countable union of proper analytic hypersurfaces in $\HH^k$, the map
\begin{equation}\label{eq:rankmap}
 \begin{aligned}
 \bigoplus_{I\subseteq\{1,\ldots,k\}} W
 &\longrightarrow \C/\Lambda_{\boldsymbol\tau,\Q},\\
 ([a_I])_I&\longmapsto
 \left[\sum_I a_I\prod_{i\in I}\tau_i\right]
 \end{aligned}
\end{equation}
is well-defined and injective. For $k=0$ it is simply the inclusion of $W$.
\end{proposition}
\begin{proof}
Changing $a_I$ by an element of $\Lambda_\Q$ changes the displayed sum by an element of~\eqref{eq:periodk}; hence the map is well-defined and $\Q$-linear. A countable-dimensional vector space over $\Q$ has countably many elements. Choose one representative $a_w\in\C$ of each $w\in W$, with $a_0=0$.

For a nonzero tuple $(w_I)_I$ and a tuple $(\lambda_I)_I$ with $\lambda_I\in\Lambda_\Q$, consider
\[
 P_{(w_I),(\lambda_I)}(\boldsymbol\tau)
   =\sum_I(a_{w_I}-\lambda_I)\prod_{i\in I}\tau_i.
\]
If $w_I\ne0$, then $a_{w_I}\notin\Lambda_\Q$, so that coefficient is nonzero. Distinct subsets give distinct monomials, and therefore this polynomial is not identically zero. There are only countably many pairs of tuples. Excluding their zero sets proves injectivity, and the Baire category theorem gives a dense nonempty complement. When $k=0$, these nonzero constant polynomials have no zeros.
\end{proof}

The original rank-preservation map $W\to\C/\Lambda_{\boldsymbol\tau,\Q}$ is the restriction to the summand indexed by $I=\varnothing$. When $W$ is spanned by reduced Abel--Jacobi values of algebraic cycles on $Y$, the summands in~\eqref{eq:rankmap} are realized by the following loop constructions.

For each elliptic factor, let $\gamma_{j,0}=\gamma_j$ be its marked $a$-loop and let $\gamma_{j,1}$ be its marked $b$-loop, both based at $0$. Thus
\[
 \int_{\gamma_{j,0}}du_j=1,\qquad
 \int_{\gamma_{j,1}}du_j=\tau_j.
\]
Suppose $Z_1,\ldots,Z_d$ are integral homologically trivial $1$-cycles on the fixed $Y$, with $\partial U_\ell=Z_\ell$, and write $A_\ell=\int_{U_\ell}\Omega$. For $k\geq1$ and $I\subseteq\{1,\ldots,k\}$, define $\epsilon_j(I)=1$ when $j\in I$ and $0$ otherwise. The family
\begin{equation}\label{eq:abfamily}
 (t_1,\ldots,t_k)\longmapsto
 \prod_{j=1}^k\bigl([\gamma_{j,\epsilon_j(I)}(t_j)]-[0]\bigr)\times Z_\ell
\end{equation}
factors through $S^k$ exactly as in~\eqref{eq:smashfamily}. Denote its integral Lawson class by $\alpha_{\boldsymbol\tau,I,\ell}\in L_1H_{k+2}(X_k)$. For $k=0$, the class indexed by $I=\varnothing$ is understood to be $[Z_\ell]\in L_1H_2(Y)$.

\begin{theorem}[Amplification by both elliptic loops]\label{thm:amplify}
If $[A_1],\ldots,[A_d]$ are linearly independent in $\C/\Lambda_\Q$, then for very general $\boldsymbol\tau\in\HH^k$ the $d\,2^k$ classes
\[
 \Reg_{1,k}^{X_k}(\alpha_{\boldsymbol\tau,I,\ell}\otimes1),
 \qquad I\subseteq\{1,\ldots,k\},\quad 1\leq\ell\leq d,
\]
are $\Q$-linearly independent. Every $\alpha_{\boldsymbol\tau,I,\ell}$ has zero integral Betti class. The same lower bound holds in $L_pH_{2p+k}$ after pullback to $X_k\times\PP^{p-1}$, for every $p\geq1$.
\end{theorem}
\begin{proof}
The trace and a bounding chain of~\eqref{eq:abfamily} are respectively
\[
 \left(\prod_{j=1}^k\gamma_{j,\epsilon_j(I)}\right)\times Z_\ell,
 \qquad
 (-1)^k\left(\prod_{j=1}^k\gamma_{j,\epsilon_j(I)}\right)\times U_\ell.
\]
Thus the Betti class vanishes. Combining the geometric boundary sign $(-1)^k$ with the current-to-chain sign $(-1)^M=(-1)^k$, evaluation of the regulator on $\Omega_k$ gives
\begin{equation}\label{eq:abperiod}
  A_\ell\prod_{i\in I}\tau_i
        \pmod{\Lambda_{\boldsymbol\tau,\Q}}.
\end{equation}
The denominator remains the full group~\eqref{eq:periodk}; it does not depend on which loops were used in the family. Apply Proposition~\ref{prop:rank} to the span of $[A_1],\ldots,[A_d]$. The images in~\eqref{eq:abperiod} are the images of a basis of the direct sum in~\eqref{eq:rankmap} and are therefore independent. Smooth pullback and the section of $X_k\times\PP^{p-1}\to X_k$ preserve this independence, by Proposition~\ref{prop:operations}.
\end{proof}

\begin{proof}[The independent-moduli rank bound]
Use the two cycles on one fixed $Y$ supplied by Corollary~\ref{cor:twoclasses}, and apply Theorem~\ref{thm:amplify} with $d=2$. The dimension is $(k+3)+(p-1)=p+k+2$ and the lower bound is $2^{k+1}$.
\end{proof}

Proposition~\ref{prop:rank} also preserves an infinite-dimensional countable subspace \emph{if such a subspace is already known in the reduced Abel--Jacobi image of a fixed threefold $Y$}. The two explicit cycles only supply a two-dimensional subspace. Infinite rank of a Griffiths group alone does not imply infinite rank of this projected image.

\subsection{One elliptic curve in every homotopy degree}
The generic choice of a tuple in $\HH^k$ does not automatically permit restriction to the diagonal $\tau_1=\cdots=\tau_k$. On the diagonal, different subset monomials coincide. A separate one-variable argument gives a uniform alternative.

\begin{theorem}[A uniform elliptic modulus]\label{thm:uniformE}
Fix a smooth projective threefold $Y$ and integral homologically trivial $1$-cycles $Z_1,\ldots,Z_d$ with bounding chains $U_1,\ldots,U_d$. Suppose the values $A_\ell=\int_{U_\ell}\Omega$, for the fixed $\Omega$ and period group $\Lambda_\Q$ of this section, have independent classes in $\C/\Lambda_\Q$. There is a countable set $\mathcal E\subset\HH$ such that every $\tau\in\HH\setminus\mathcal E$ has the following property simultaneously for all $p\geq1$ and $k\geq0$: the variety
\[
 E_\tau^k\times Y\times\PP^{p-1}
\]
has at least $d(k+1)$ independent rational Lawson--Deligne images of integral homologically trivial classes in $L_pH_{2p+k}$. For the two classes in Corollary~\ref{cor:twoclasses}, this lower bound is $2(k+1)$.
\end{theorem}
\begin{proof}
For $E_\tau^k\times Y$, the complete period group~\eqref{eq:periodk} becomes
\[
 \Lambda_{(\tau,\ldots,\tau),\Q}
       =\sum_{j=0}^k\tau^j\Lambda_\Q.
\]
For each $0\leq j\leq k$ and each $\ell$, choose the family~\eqref{eq:abfamily} with $b$-loops in the first $j$ factors and $a$-loops in the others. Its regulator value is $A_\ell\tau^j$.

A nonzero rational relation among these $d(k+1)$ values modulo the entire period group would imply
\[
 \sum_{j=0}^k
 \left(\sum_{\ell=1}^d a_{\ell,j}A_\ell-\lambda_j\right)\tau^j=0,
 \qquad a_{\ell,j}\in\Q,\quad\lambda_j\in\Lambda_\Q,
\]
where the array $(a_{\ell,j})$ is nonzero. Independence of the $[A_\ell]$ shows that the polynomial is nonzero: a column with a nonzero coefficient array has coefficient outside $\Lambda_\Q$ before subtracting $\lambda_j$. There are countably many such polynomials, and each has finitely many roots. Take their roots in $\HH$, for every $k\geq0$, as $\mathcal E$. This is still a countable set. Pullback along the projective-space factor then gives every $p$ without any further exceptional parameters.
\end{proof}

This argument selects one fixed $E_\tau$ for all $k$; it does not certify a prescribed modulus. The bound $d(k+1)$ counts distinct powers of $\tau$, whereas the bound $d\,2^k$ in Theorem~\ref{thm:amplify} counts distinct monomials in independently chosen moduli.

\section{Prescribed elliptic factors and complex multiplication}\label{sec:CM}
The preceding section fixes $Y$ and then avoids exceptional elliptic moduli. This section instead fixes the moduli. The two results have different quantifiers, and neither follows by specializing a generic-parameter statement without further proof.

\subsection{The full period coefficient space}
For a fixed $\boldsymbol\tau\in\HH^k$, retain $S_{\boldsymbol\tau}$ and $\delta(\boldsymbol\tau)$ from~\eqref{eq:Sintro}. On $X_{k,t}=T_{\boldsymbol\tau}\times Y_t$ the full period group on the top form $\Omega_k$ is
\begin{equation}\label{eq:fixedfullperiod}
 \Lambda_{\boldsymbol\tau,\Q}(t)
       =S_{\boldsymbol\tau}\Lambda_\Q(t).
\end{equation}
This follows from Lemma~\ref{lem:kperiods}, not just from the selected loops. For the $d$ cycle families $Z_{\ell,t}$ in Theorem~\ref{thm:coefficientPF}, let $\LoopIm_{1,k}(X_{k,t})$ be the $\Q$-span of the images of every class $\alpha_{\boldsymbol\tau,I,\ell}$ in~\eqref{eq:abfamily}. For $p>1$, let $\LoopIm_{p,k}$ be the image of this span under pullback to $X_{k,t}\times\PP^{p-1}$.

\begin{theorem}[Exact rank of the loop-generated image]\label{thm:fixedrank}
Assume the family, cycle classes, and nowhere-vanishing form satisfy the hypotheses of Theorem~\ref{thm:coefficientPF}, including $h^{3,0}(Y_t)=1$ and $\C$-linear independence of the $g_\ell$. For each prescribed $\boldsymbol\tau\in\HH^k$ and all but countably many $t\in U$,
\begin{equation}\label{eq:fixedrank}
 \LoopIm_{p,k}(X_{k,t}\times\PP^{p-1})
       \simeq S_{\boldsymbol\tau}^{\oplus d},
 \qquad \dim_\Q\LoopIm_{p,k}=d\,\delta(\boldsymbol\tau).
\end{equation}
All relations among the displayed $d2^k$ regulator images are exactly the relations among the period monomials, separately for each cycle family. This does not assert that the same relations hold among the Lawson classes before applying the regulator.
\end{theorem}
\begin{proof}
Set $\Omega_{k,t}=du_1\wedge\cdots\wedge du_k\wedge\Omega_t$. Since $h^{k+3,0}(X_{k,t})=1$, evaluation on $\Omega_{k,t}$ identifies the entire secondary target of $\Reg_{1,k}$ with $\C/\Lambda_{\boldsymbol\tau,\Q}(t)$. Each loop-generated image evaluates as
\[
 T_\ell(t)m_I,
 \qquad m_I=\prod_{i\in I}\tau_i.
\]
The span of all arrays of coefficients obtained from the monomials is exactly $S_{\boldsymbol\tau}^{\oplus d}$. Equation~\eqref{eq:fixedfullperiod} identifies its evaluation map with~\eqref{eq:coeffsep}. Theorem~\ref{thm:coefficientPF} makes this map injective for all but countably many $t$; by the definition of $\LoopIm$ it is surjective onto that span. This proves both the exact dimension and the assertion about regulator relations. Pullback along the projective-space projection has a Deligne left inverse, so the same exact rank holds for every $p$.
\end{proof}

\begin{proof}[Proof of Theorem~\ref{thm:prescribedintro}]
Corollary~\ref{cor:twoclasses} establishes independence of the two inhomogeneous functions over $\C$, before selecting any fiber. Apply Theorem~\ref{thm:fixedrank} with $d=2$. Its simultaneous assertion follows from the last sentence of Theorem~\ref{thm:coefficientPF}. Every bounding chain was explicitly specified in~\eqref{eq:abfamily}--\eqref{eq:abperiod}; hence all Betti classes vanish integrally.
\end{proof}

\begin{corollary}[A common fiber for algebraic period parameters]\label{cor:allalgebraic}
Let $\overline{\Q}\subset\C$ denote the field of algebraic numbers. Here algebraicity refers to the marked period parameters $\tau_j$. There is one smooth mirror-quintic fiber $Y$ on which the classes of $A_1$ and $A_2$ are linearly independent over $\overline{\Q}$ modulo $\overline{\Q}\Lambda_\Q$. Consequently this same $Y$ gives the exact formula $\dim_\Q\LoopIm_{p,k}=2\delta(\boldsymbol\tau)$ for every $k\geq0$, every $p\geq1$, and every tuple $\boldsymbol\tau\in(\overline{\Q}\cap\HH)^k$.
\end{corollary}
\begin{proof}
Apply Theorem~\ref{thm:coefficientPF} to the countable coefficient field $S=\overline{\Q}$ and choose a fiber outside its countable exceptional set. For any algebraic tuple, $S_{\boldsymbol\tau}\subset\overline{\Q}$ and $S_{\boldsymbol\tau}\Lambda_\Q\subset\overline{\Q}\Lambda_\Q$. A relation modulo the smaller period group would therefore give a relation modulo the larger one, which is impossible unless every coefficient of the $A_\ell$ is zero. The proof of Theorem~\ref{thm:fixedrank} then applies simultaneously to all these tuples. The parameter of the chosen mirror-quintic fiber is not asserted to be algebraic.
\end{proof}

\subsection{One repeated prescribed modulus}
\begin{corollary}[Rank saturation for a fixed modulus]\label{cor:repeated}
Fix $\tau\in\HH$. A single $Y$ can be chosen so that the cycles $Z_1$ and $Z_2$ give, simultaneously for every $k\geq0$ and $p\geq1$, a loop-generated subspace on $E_\tau^k\times Y\times\PP^{p-1}$ of dimension
\begin{equation}\label{eq:repeatedrank}
 2\dim_\Q\operatorname{span}_\Q\{1,\tau,\ldots,\tau^k\}
 =\begin{cases}
 2\min(k+1,e),&[\Q(\tau):\Q]=e<\infty,\\
 2(k+1),&\tau\text{ is transcendental}.
 \end{cases}
\end{equation}
\end{corollary}
\begin{proof}
For a repeated modulus, the monomials are precisely the powers $1,\tau,\ldots,\tau^k$. Their span has dimension $\min(k+1,e)$ for an algebraic number of degree $e$, by the minimal polynomial, and dimension $k+1$ when $\tau$ is transcendental. Apply Theorem~\ref{thm:fixedrank} for all $k$ and take the countable union of its excluded fiber sets. Changing $p$ imposes no further condition.
\end{proof}

\begin{example}[The square elliptic curve]\label{ex:squareE}
Take the fully prescribed curve $E_i=\C/(\Z+i\Z)$. Then
\[
 S_{(i,\ldots,i)}=\Q(i)\quad(k\geq1),\qquad
 \Lambda_{(i,\ldots,i),\Q}(t)=\Lambda_\Q(t)+i\Lambda_\Q(t).
\]
For a common suitable $Y$, the entire specified loop-generated image is exactly four-dimensional for every $k\geq1$: representatives may be chosen with periods $A_1,iA_1,A_2,iA_2$. For $k=2$ the four loop choices for each cycle have periods
\[
 A_\ell,\quad iA_\ell,\quad iA_\ell,\quad -A_\ell.
\]
Thus there are eight displayed images but only four independent ones for the two cycles. Counting loop choices without their period relations would give the wrong rank. For $k=0$, the rank is two.
\end{example}

\subsection{Quadratic moduli and square-class relations}
\begin{lemma}[Quadratic period monomials]\label{lem:multiquad}
Let $\ell_1,\ldots,\ell_k$ be distinct primes and put $\tau_j=i\sqrt{\ell_j}$. Then the $2^k$ products $m_I=\prod_{i\in I}\tau_i$ are linearly independent over $\Q$. Moreover, each $E_{\tau_j}$ has an endomorphism induced by multiplication by $\tau_j$.
\end{lemma}
\begin{proof}
The squareclasses of $-\ell_1,\ldots,-\ell_k$ are independent in $\Q^*/\Q^{*2}$: in a nonempty product the valuation at any prime appearing in it is odd. The elementary multiquadratic extension lemma therefore gives
\[
 [\Q(\sqrt{-\ell_1},\ldots,\sqrt{-\ell_k}):\Q]=2^k
\]
and the products of the square roots form a basis. For completeness, the lemma follows inductively as follows. In an existing multiquadratic extension, any element whose square lies in $\Q^*$ is an eigenvector for all sign automorphisms. The simultaneous eigenspaces are the one-dimensional spaces spanned by products of the existing square roots. Thus a new rational squareclass independent of the earlier ones cannot acquire a square root in that extension, so adjoining it doubles the degree and extends the product basis. This proves the induction.

Finally $\tau_j\cdot1=\tau_j$ and $\tau_j^2=-\ell_j$. Consequently multiplication by $\tau_j$ preserves the lattice $\Z+\tau_j\Z$ and gives a noninteger elliptic endomorphism. This is complex multiplication by an order containing $\Z[\sqrt{-\ell_j}]$.
\end{proof}

\begin{corollary}[Square-class rank calculation]\label{cor:squareclasses}
More generally, prescribe positive integers $b_1,\ldots,b_k$ and set $\tau_j=i\sqrt{b_j}$. Define
\[
 \rho=\dim_{\mathbb F_2}\langle[-b_1],\ldots,[-b_k]\rangle
       \quad\text{in }\Q^*/\Q^{*2}.
\]
On the common threefold $Y$ of Corollary~\ref{cor:allalgebraic}, the exact loop-image rank for the two cycles is $2^{\rho+1}$, in every positive cycle dimension obtained by the projective-space factor.
\end{corollary}
\begin{proof}
Because $\tau_j^2=-b_j\in\Q$, the monomial space $S_{\boldsymbol\tau}$ is closed under multiplication and contains $\Q$ and all $\tau_j$. It is a finite-dimensional domain over $\Q$, hence is the field $\Q(\tau_1,\ldots,\tau_k)$. Choose an independent basis among the indicated squareclasses. The multiquadratic argument in Lemma~\ref{lem:multiquad} gives degree $2^\rho$ for their square roots. Every remaining square root is a rational multiple, up to sign, of a product of those basis square roots. Thus $\delta(\boldsymbol\tau)=2^\rho$, and Theorem~\ref{thm:fixedrank} proves the formula.
\end{proof}

\begin{example}[Distinct fields can still have a relation]\label{ex:dependentCM}
For the four prescribed moduli
\[
 (\tau_1,\tau_2,\tau_3,\tau_4)
     =(i\sqrt2,i\sqrt3,i\sqrt5,i\sqrt{30}),
\]
we have $\tau_4=-\tau_1\tau_2\tau_3$. The squareclasses have rank three because $(-2)(-3)(-5)(-30)=30^2$, while the first three are independent. Thus the $32$ displayed loop images of the two cycles span a space of dimension exactly $16$, not $32$. Pairwise distinct imaginary quadratic fields alone do not imply independence of all period monomials. This exact relation complements the independent-prime example and the repeated-modulus example.
\end{example}

\begin{proof}[Proof of Corollary~\ref{cor:CMintro}]
Lemma~\ref{lem:multiquad} gives $\delta(\tau_1,\ldots,\tau_k)=2^k$. Theorem~\ref{thm:fixedrank} then gives $2^{k+1}$ independent regulator images. There are only countably many finite initial segments of the prescribed prime sequence, so one $Y$ works for every $k$. The repeated imaginary quadratic case is Corollary~\ref{cor:repeated} with $e=2$. One can include these repeated tuples in the same countable collection and select $Y$ simultaneously.
\end{proof}

\begin{example}[Period coefficient spaces and their dimensions]\label{ex:CMtable}
The following table computes the relevant finite subspace ranks; $Y$ is chosen outside the common countable exceptional set.
\begin{center}
\begin{tabular}{@{}clcc@{}}
\toprule
$k$ & Prescribed moduli & $\delta(\boldsymbol\tau)$ & $\dim_\Q\LoopIm_{p,k}$\\
\midrule
$0$ & Empty tuple & $1$ & $2$\\
$1$ & $(i\sqrt2)$ & $2$ & $4$\\
$2$ & $(i\sqrt2,i\sqrt3)$ & $4$ & $8$\\
$3$ & $(i\sqrt2,i\sqrt3,i\sqrt5)$ & $8$ & $16$\\
$k\geq1$ & $(i,\ldots,i)$ & $2$ & $4$\\
\bottomrule
\end{tabular}
\end{center}
For $k=2$, the period coefficient basis is $1,i\sqrt2,i\sqrt3,-\sqrt6$. The geometric dimension is always $p+k+2$. The elliptic parameters are now specified numbers; the mirror-quintic fiber is still chosen outside a countable exceptional set, rather than at a specified number-field point. The computations are refinements of the product-detection mechanism of \cite{HuAJ}; no claim that complex multiplication or the multiquadratic basis lemma is new is intended.
\end{example}

\section{Exceptional-divisor detection on rational varieties}\label{sec:blowup}

\subsection{Known decompositions and their regulator-image consequence}
Let $V\subset X$ be a smooth center of codimension $r\geq2$, and let $b:B=\Bl_VX\to X$. Write $j:E\hookrightarrow B$ and $\pi:E\to V$, and let $\xi=c_1(\mathcal O_E(1))$ with $N_{E/B}=\mathcal O_E(-1)$. The Lawson blow-up theorem \cite{HuB,HuLi} is
\begin{equation}\label{eq:knownLawsonsplit}
 L_pH_{2p+k}(B,\Q)\simeq L_pH_{2p+k}(X,\Q)
 \oplus\bigoplus_{i=1}^{r-1}L_{p-i}H_{2(p-i)+k}(V,\Q).
\end{equation}
The summand maps are $b^*$ and
\begin{equation}\label{eq:alle}
 \mathfrak e_i(\alpha)=j_*(\xi^{r-1-i}\cap\pi^*\alpha),
 \qquad 1\leq i\leq r-1.
\end{equation}
For $a<0$, use the conventional extension $L_aH_b(-,\Q)=H_b(-,\Q)$ with its identity cycle map; singular homology in negative degree is zero. These terms have zero cycle-map kernel. This convention is used only to state the standard splitting without truncation.

Put $M_B=2(\dim X-p)-k$ and $q_B=\dim X-p-k-1$. The corresponding Deligne splitting is
\begin{equation}\label{eq:knownDelignesplit}
 \HD^{M_B}(B,\Q(q_B))\simeq
 \HD^{M_B}(X,\Q(q_B))\oplus
 \bigoplus_{i=1}^{r-1}
 \HD^{M_B-2(r-i)}(V,\Q(q_B-r+i)).
\end{equation}
Its exceptional map is $j_*(\xi_{\mathcal D}^{r-1-i}\smile\pi^*(-))$. This is the usual Deligne blow-up decomposition; it also follows by applying the filtered-current cone to the filtered and Betti blow-up decompositions. These formulas, including Deligne decompositions used for configuration spaces, predate this paper \cite{EV,HuLiFM}.

\begin{theorem}[Splitting of the detected image]\label{thm:fullblowup}
On the homologically trivial Lawson subgroups, the decompositions~\eqref{eq:knownLawsonsplit}--\eqref{eq:knownDelignesplit} identify the regulator with the direct sum of the regulator on $X$ and the shifted regulators on the center $V$ in nonnegative cycle dimension. The conventionally extended negative-dimensional summands have zero cycle-map kernel. Consequently
\[
 \Img_{p,k}(B)\simeq\Img_{p,k}(X)
            \oplus\bigoplus_{i=1}^{r-1}\Img_{p-i,k}(V).
\]
\end{theorem}
\begin{proof}
The action $b^*$ is compatible by Theorem~\ref{thm:corr}. For $p-i\geq0$, the map $\mathfrak e_i$ is a composition of smooth pullback, powers of an ordinary first Chern class, and proper pushforward, so Proposition~\ref{prop:operations} gives
\[
 \Reg_{p,k}^{B}(\mathfrak e_i\alpha)
   =j_*(\xi_{\mathcal D}^{r-1-i}\smile\pi^*
                           \Reg_{p-i,k}^{V}(\alpha)).
\]
If $m=\dim V=\dim X-r$, the source indices are
\[
 M_V=2(m-(p-i))-k=M_B-2(r-i),\qquad
 q_V=m-(p-i)-k-1=q_B-r+i.
\]
These are exactly the indices in~\eqref{eq:knownDelignesplit}. The Betti realization respects the same direct sums. Therefore a Lawson class in~\eqref{eq:knownLawsonsplit} is homologically trivial exactly when each component is; restricting the diagonal regulator identity to these kernels proves the image decomposition. Components with nonpositive cycle dimension have zero kernel. Components with $q_V\leq0$ have zero secondary image by the explicit cone convention. Thus no unavailable negative-dimensional regulator is being used.
\end{proof}

\begin{corollary}[Projective bundles]\label{cor:PBimage}
If $P\to X$ is a projective bundle with fiber $\PP^{r-1}$, then the usual projective-bundle correspondences identify
\[
 \Img_{p,k}(P)\simeq\bigoplus_{i=0}^{r-1}\Img_{p-i,k}(X).
\]
\end{corollary}
\begin{proof}
Use the projective-bundle formulas of \cite{FG,HuLi} on Lawson homology and the corresponding formulas on Deligne cohomology. Their summand maps are powers of $c_1(\mathcal O_P(1))$ composed with smooth pullback, so their compatibility follows from Proposition~\ref{prop:operations}. Taking cycle-map kernels and then regulator images gives the displayed identity, as in Theorem~\ref{thm:fullblowup}.
\end{proof}
This proves Theorem~\ref{thm:splitintro}. These are compatibility consequences of known decompositions; no new Lawson projective-bundle or blow-up theorem is asserted.

\subsection{One exceptional summand with an explicit left inverse}
Let $V$ be a smooth projective variety of dimension $m$ embedded in $\PP^N$ with codimension $r=N-m\geq2$. Form
\[
 B=\Bl_V\PP^N,\qquad E=\PP(N_{V/\PP^N}),
 \qquad j:E\hookrightarrow B,\quad\pi:E\to V.
\]
Use the projective-bundle convention in which
$N_{E/B}=\mathcal O_E(-1)$, and put $\xi=c_1(\mathcal O_E(1))$.

\begin{theorem}[A regulator-preserving exceptional injection]\label{thm:blowup}
For $\alpha\in L_pH_{2p+k}(V,\Q)$ define
\begin{equation}\label{eq:beta}
 \mathfrak e(\alpha)=j_*\big(\xi^{r-2}\cap\pi^*\alpha\big)
 \in L_{p+1}H_{2p+k+2}(B,\Q).
\end{equation}
Put $M_V=2(m-p)-k$ and $q_V=m-p-k-1$. If $q_V\geq1$, then
\begin{equation}\label{eq:blowupreg}
 \Reg_{p+1,k}^B(\mathfrak e(\alpha))
 =j_*\left((\xi_{\mathcal D})^{r-2}\smile
                       \pi^*\Reg_{p,k}^V(\alpha)\right)
\end{equation}
in $\HD^{M_V+2r-2}(B,\Q(q_V+r-1))$, where
$\xi_{\mathcal D}=c_1^{\mathcal D}(\mathcal O_E(1))$.
The Deligne operation on the right is injective. In particular, a detected homologically trivial class on $V$ gives a detected homologically trivial class on $B$.
\end{theorem}
\begin{proof}
The map $\pi$ is smooth of relative dimension $r-1$. Its pullback raises the cycle dimension by $r-1$; intersection with $\xi^{r-2}$ lowers it by $r-2$. The resulting cycle dimension is $p+1$, and the homotopy degree remains $k$. Applying Proposition~\ref{prop:operations} successively proves~\eqref{eq:blowupreg} and the stated shifts.

To prove injectivity, let
\[
 \mathfrak e_{\mathcal D}(u)=j_*\big(\xi_{\mathcal D}^{r-2}\smile\pi^*u\big).
\]
The usual Deligne Gysin self-intersection and projective-bundle formulas give
\[
 j^*j_*(v)=c_1^{\mathcal D}(N_{E/B})\smile v
         =-\xi_{\mathcal D}\smile v,
 \qquad \pi_*(\xi_{\mathcal D}^{r-1})=1.
\]
Consequently
\begin{equation}\label{eq:leftinverse}
 -\pi_*j^*\mathfrak e_{\mathcal D}(u)
 =\pi_*\big(\xi_{\mathcal D}^{r-1}\smile\pi^*u\big)=u.
\end{equation}
Only the ordinary cohomological pullback $j^*$ in Deligne cohomology is used in this argument; no claim about analytic pullback of arbitrary currents is needed. Naturality of the Betti class under the same geometric operations preserves homological triviality.
\end{proof}

The operation~\eqref{eq:beta} is the first exceptional summand of the usual Lawson blow-up formula \cite{HuB,HuLi}. The additional statement here is the explicit compatibility~\eqref{eq:blowupreg}, together with a left inverse in the \emph{secondary} target.

\begin{proposition}[Detected classes on rational varieties]\label{thm:rationalintro}
For every $k\geq0$ there is a smooth projective rational $(2k+7)$-fold carrying a detected integral class in $L_2H_{k+4}$.
\end{proposition}
\begin{proof}
Take $V=X_k$ from Section~\ref{sec:elliptic}, so $m=k+3$, $p=1$, and $\alpha$ is detected with $M_V=k+4$, $q_V=1$. A smooth projective $m$-fold embeds in $\PP^{2m+1}$. Indeed, begin with a sufficiently large very ample embedding. The closure of the union of secant and tangent lines has dimension at most $2m+1$, so a general center of projection to $\PP^{2m+1}$ misses it. The resulting map is injective and immersive, hence a closed immersion because the source is projective. Thus choose $N=2k+7$ and $r=k+4$.

The blow-up $B=\Bl_V\PP^{2k+7}$ is smooth, projective, and rational, since its blow-down is birational to projective space. Theorem~\ref{thm:blowup} gives the desired integral class in $L_2H_{k+4}(B)$. The Deligne target is
\[
 \HD^{3k+10}(B,\Q(k+4)).
\]
For $k=1$, these indices are $N=9$, $L_2H_5$, and $\HD^{13}(-,\Q(5))$, as claimed.
\end{proof}

\begin{corollary}\label{cor:rationalallp}
For every $p\geq2$ and $k\geq0$, there is a smooth rational projective variety of dimension $p+2k+5$ carrying a class in $L_pH_{2p+k}$ with zero Betti class and nonzero rational Lawson--Deligne class.
\end{corollary}
\begin{proof}
Multiply the variety in Proposition~\ref{thm:rationalintro} by $\PP^{p-2}$ and pull back the detected class. A section proves injectivity of the corresponding Deligne pullback, exactly as in the proof of Theorem~\ref{thm:allintro}.
\end{proof}

\begin{remark}[What the rational example means]
Nonzero secondary Lawson--Deligne classes are not, in general, obstructions to rationality. Blow-ups can import the relevant off-diagonal Hodge structure and secondary cycle classes from their centers. This does not contradict Corollary~\ref{cor:Tate}: a smooth rational variety need not have Hodge--Tate cohomology. No minimality assertion is made for the dimensions in the above construction.
\end{remark}

\subsection{Disjoint centers and arbitrary finite rank}
\begin{lemma}[Disjoint projective translates]\label{lem:translates}
Let $V\subset\PP^N$ be a smooth projective variety of dimension $m$, with $2m<N$. For every integer $s\geq1$, there are projective automorphisms carrying $V$ to pairwise disjoint subvarieties $V_1,\ldots,V_s\subset\PP^N$.
\end{lemma}
\begin{proof}
Let $G=\operatorname{PGL}_{N+1}(\C)$. For a fixed copy $V_a$, consider the incidence condition $gx=y$ with $x\in V$ and $y\in V_a$. For each $(x,y)$, the set of such $g$ has dimension $\dim G-N$, since $G$ acts transitively on $\PP^N$. The full incidence variety therefore has dimension $\dim G+2m-N<\dim G$. Its image in $G$ is closed, because $V\times V_a$ is projective, and is proper by the dimension bound. Since $G$ is irreducible, the complement of a finite union of these proper closed subsets is nonempty. Choosing a translate in this complement at each step proves the assertion.
\end{proof}

For such a collection, put
\[
 B_s=\Bl_{V_1\sqcup\cdots\sqcup V_s}\PP^N.
\]
The centers are smooth and disjoint, so $B_s$ is smooth, projective, and rational. Write $j_a:E_a\hookrightarrow B_s$ and $\pi_a:E_a\to V_a$ for its exceptional divisors and projections. Put $r=N-m$, $\xi_a=c_1(\mathcal O_{E_a}(1))$, and $\xi_{a,\mathcal D}=c_1^{\mathcal D}(\mathcal O_{E_a}(1))$. The notation extends that of the single-center construction.

\begin{theorem}[Independent exceptional images]\label{thm:disjoint}
Assume $r\geq2$, put $M_V=2(m-p)-k$, and suppose $q_V=m-p-k-1\geq1$. On each center use the operation
\[
 e_a(\alpha)=j_{a*}(\xi_a^{r-2}\cap\pi_a^*\alpha).
\]
Then the sum of the corresponding Deligne operations
\begin{equation}\label{eq:disjointmap}
 (u_1,\ldots,u_s)\longmapsto
 \sum_{a=1}^s j_{a*}
       (\xi_{a,\mathcal D}^{r-2}\smile\pi_a^*u_a)
\end{equation}
is injective from
\[
 \bigoplus_{a=1}^s\HD^{M_V}(V_a,\Q(q_V))
 \quad\text{to}\quad
 \HD^{M_V+2r-2}(B_s,\Q(q_V+r-1)).
\]
It identifies independent detected subspaces on the centers with independent detected subspaces of the Lawson--Deligne image on $B_s$.
\end{theorem}
\begin{proof}
For each $a$ the regulator identity is exactly~\eqref{eq:blowupreg}. The self-intersection formula gives the diagonal identity of~\eqref{eq:leftinverse}. For $a\ne b$, the divisors $E_a,E_b$ are disjoint, so restriction to $E_a$ of a Gysin class supported on $E_b$ is zero. Equivalently, the pullback--Gysin base-change formula has empty fiber product. Consequently, with $e_{b,\mathcal D}$ denoting the $b$-th summand in~\eqref{eq:disjointmap},
\begin{equation}\label{eq:orthogonalleftinverse}
 -\pi_{a*}j_a^*e_{b,\mathcal D}(u)
 =\begin{cases}u,&a=b,\\0,&a\ne b.\end{cases}
\end{equation}
These left inverses recover each component separately and prove injectivity. Compatibility with the Betti class, already established for each $e_a$, preserves homological triviality. Thus any rational relation among the resulting regulator images would pull back to a relation on each center, proving the final assertion.
\end{proof}

\begin{proof}[Proof of Theorem~\ref{thm:rankintro}]
First let $p=2$. Choose the prescribed complex multiplication tuple and a common threefold $Y$ carrying the two independent classes, as in Corollary~\ref{cor:CMintro}. The variety $V=X_k$ has dimension $m=k+3$ and supplies $2^{k+1}$ independent detected classes in $L_1H_{k+2}(V)$, with $M_V=k+4$ and $q_V=1$. Embed $V$ in $\PP^{2m+1}=\PP^{2k+7}$. Lemma~\ref{lem:translates} produces $s$ disjoint copies, all with the same detected subspace.

Theorem~\ref{thm:disjoint} embeds the direct sum of these subspaces into the image in
\[
 \HD^{3k+10}(B_s,\Q(k+4))
\]
of homologically trivial integral classes in $L_2H_{k+4}(B_s)$. The specified exceptional subspace has dimension exactly $s\,2^{k+1}$; the entire regulator image has dimension at least this number. For general $p\geq2$, take the product of this $B_s$ with $\PP^{p-2}$ and pull back the classes. The projection has a section, so the regulator images remain independent. Calling this product $B_s$ in the stated bidegree gives
\[
 \dim B_s=2k+7+p-2=p+2k+5.
\]
All cycle operations used here preserve integral classes and their zero integral Betti class.
\end{proof}

\begin{example}[Rational ninefolds]\label{ex:manyninefolds}
For every $s\geq1$, the construction with $k=1$ gives a smooth rational ninefold $B_s$ and at least $4s$ independent images
\[
 \Reg_{2,1}^{B_s}(\beta_{a,I,\ell}\otimes1)
       \in\HD^{13}(B_s,\Q(5)),
 \qquad \Phi_{2,5}(\beta_{a,I,\ell})=0.
\]
Here $\beta_{a,I,\ell}=e_a(\alpha_{I,\ell}^{(a)})$, where $\alpha_{I,\ell}^{(a)}$ is the copy on $V_a$ of the loop class in~\eqref{eq:abfamily}. The indices $1\leq a\leq s$, $I\subseteq\{1\}$, and $\ell=1,2$ specify the center, loop choice, and algebraic cycle. The dimension of $B_s$ remains nine as $s$ grows. The varieties vary with $s$, so this is not a proof of infinite rank on a single ninefold.
\end{example}

\subsection{Several detected bidegrees on the same rational variety}
\begin{corollary}[Simultaneous detection in the exceptional summands]\label{cor:allslots}
Fix $k\geq0$ and $s\geq1$. For the same rational $(2k+7)$-fold $B_s=\Bl_{V_1\sqcup\cdots\sqcup V_s}\PP^{2k+7}$, the image $\Img_{p,k}(B_s)$ contains a specified $s2^{k+1}$-dimensional subspace for every $2\leq p\leq k+4$. These subspaces are realized by integral homologically trivial classes in
\[
 L_pH_{2p+k}(B_s,\Q),\qquad 2\leq p\leq k+4.
\]
The corresponding Deligne targets are
\[
 \HD^{3k+14-2p}(B_s,\Q(k+6-p)).
\]
\end{corollary}
\begin{proof}
Each center $V_a=X_k$ has dimension $k+3$ and a $2^{k+1}$-dimensional detected subspace in $L_1H_{k+2}$. Its codimension is $r=k+4$. For a fixed $p$ in the displayed range, take $i=p-1\in\{1,\ldots,r-1\}$ in~\eqref{eq:alle}. Then
\[
 \beta_{a,p}=j_{a*}(\xi_a^{k+4-p}\cap\pi_a^*\alpha_a)
           \in L_pH_{2p+k}(B_s).
\]
The full Deligne decomposition makes each exceptional summand map injective, and the images associated with disjoint centers are independent. Direct substitution of $n=2k+7$ in~\eqref{eq:indices} gives the target indices. The construction using the first exceptional summand is therefore one member of this simultaneous family. The calculation uses the known exceptional decomposition.
\end{proof}
For example, one rational ninefold from the $k=1$ construction has such subspaces in $L_2H_5$, $L_3H_7$, $L_4H_9$, and $L_5H_{11}$, with targets $\HD^{13}(\Q(5))$, $\HD^{11}(\Q(4))$, $\HD^9(\Q(3))$, and $\HD^7(\Q(2))$, respectively (the variety is suppressed in this last display).

\section{Higher Chow groups and the Bloch--KLM regulator}\label{sec:chow}
\subsection{The two constructions and their coefficients}
Let $0\leq c\leq n$, put $p=n-c$, and retain
\[
 M=2c-k,\qquad q=c-k-1.
\]
Unless otherwise stated, this section has rational coefficients and $q\geq1$. Denote the integral comparison of \cite[Construction~4.1]{FG} by
\[
 \mathrm{FG}^{\Z}_{c,k}:\CH^c(X,k)\longrightarrow
 L_{n-c}H_{2(n-c)+k}(X),
\]
and write $\mathrm{FG}_{c,k}=\mathrm{FG}^{\Z}_{c,k}\otimes\Q$ for its rationalization. The composite under consideration is
\begin{equation}\label{eq:higherchow}
 \begin{split}
 \CH^c(X,k)_\Q&\xrightarrow{\mathrm{FG}_{c,k}}
       L_{n-c}H_{2(n-c)+k}(X,\Q)\\
 &\xrightarrow{\Reg_{n-c,k}^{X}}
       \HD^{M}(X,\Q(q)).
 \end{split}
\end{equation}
We continue to denote this composite by $\widehat A_{c,k,\Q}$. In particular,
\[
 \delta_2\widehat A_{c,k,\Q}
   =\PD\,\Phi_{n-c,2(n-c)+k}\circ\mathrm{FG}_{c,k}.
\]
The existence of this composite follows from the cited comparison and Theorem~\ref{thm:mainmap}; it does not require assigning a single algebraic simplex to a sphere in a geometric realization.

Bloch's ordinary regulator, constructed in \cite{BlochReg}, is
\begin{equation}\label{eq:Blochreg}
 r^{\mathrm{Bloch}}_{c,k}:\CH^c(X,k)_\Q
        \longrightarrow\HD^M(X,\Q(c)).
\end{equation}
The higher Chow groups themselves are those of \cite{Bloch}. Kerr--Lewis--M\"uller-Stach give a current realization $r^{\mathrm{KLM}}_{c,k}$ of the ordinary regulator, not an additional regulator with a different target. Their comparison with the extension-class construction identifies
\begin{equation}\label{eq:BlochKLM}
 r^{\mathrm{KLM}}_{c,k}=r^{\mathrm{Bloch}}_{c,k}
 \qquad\text{with rational coefficients}
\end{equation}
\cite[Sections~5--7, especially Theorem~7.1]{KLM}. The real-position moving theorem used to pass from admissible currents to all rational higher Chow classes is \cite[Lemma~8.14 and its proof]{KL}. We use its rational form. The integral construction of $\ahat^X_{p,k}$ in Section~3 does not, by itself, turn that moving theorem into an integral comparison theorem.

\subsection{Normalized filtration reduction}
There is a natural comparison between the two targets, once the Tate normalization of Section~\ref{sec:cone} has been specified.
\begin{lemma}[Filtration reduction]\label{lem:filtrationreduction}
For $q\leq c$, inclusion $F^c\Dcur_X^\bullet\subseteq F^q\Dcur_X^\bullet$ induces a cochain map
\begin{equation}\label{eq:rhochain}
 \rho_{c,q}:\mathscr C_X(c)\longrightarrow\mathscr C_X(q),
 \qquad (r,\omega,a)\longmapsto(r,\omega,a),
\end{equation}
and hence a homomorphism, denoted by the same symbol,
\[
 \rho_{c,q}:\HD^M(X,\Q(c))\longrightarrow\HD^M(X,\Q(q)).
\]
On secondary subgroups it is the quotient map
\begin{equation}\label{eq:rhoJac}
 \frac{H^{M-1}(X,\C)}{F^cH^{M-1}(X,\C)+H^{M-1}(X,\Q)}
 \longrightarrow
 \frac{H^{M-1}(X,\C)}{F^qH^{M-1}(X,\C)+H^{M-1}(X,\Q)}.
\end{equation}
\end{lemma}
\begin{proof}
The inclusion of filtered subcomplexes and the identity on the other entries commute with~\eqref{eq:dcone}. The assertion about secondary subgroups follows from~\eqref{eq:Deligneexact}. In particular its kernel is
\[
 \frac{F^qH^{M-1}(X,\C)+H^{M-1}(X,\Q)}
      {F^cH^{M-1}(X,\C)+H^{M-1}(X,\Q)}.
\]
\end{proof}
In the unnormalized Deligne complexes, the transported map multiplies the coefficient and complex-form entries by $(2\pi i)^{q-c}$ and truncates the holomorphic de Rham complex at level $q$. It is \emph{not} an inclusion $\Q(c)\subseteq\Q(q)$, nor the identity between those two embedded lattices. With the normalized rational Betti entries, $\delta_2\rho_{c,q}=\delta_2$.

The comparison problem for~\eqref{eq:higherchow} is therefore the commutativity of
\begin{equation}\label{eq:comparison-square}
 \begin{array}{ccc}
 \CH^c(X,k)_\Q &\xrightarrow{\mathrm{FG}_{c,k}}&
       L_{n-c}H_{2(n-c)+k}(X,\Q)\\[3pt]
 r^{\mathrm{Bloch}}_{c,k}\big\downarrow&&
       \big\downarrow\Reg^X_{n-c,k}\\[3pt]
 \HD^M(X,\Q(c))&\xrightarrow{\rho_{c,q}}&\HD^M(X,\Q(q)).
 \end{array}
\end{equation}
We do not assume this square commutes merely because the upper arrow exists. Its verified cases and its remaining realization condition are stated below.

\subsection{Degree zero}
\begin{proposition}[Comparison for ordinary cycles]\label{prop:comparisonzero}
For $k=0$ and $q=c-1\geq1$,
\begin{equation}\label{eq:comparisonzero}
 \widehat A_{c,0,\Q}
   =\rho_{c,c-1}\circ r^{\mathrm{Bloch}}_{c,0}
   =\rho_{c,c-1}\circ r^{\mathrm{KLM}}_{c,0}.
\end{equation}
On homologically trivial cycles this is the ordinary Abel--Jacobi invariant followed by the quotient~\eqref{eq:rhoJac}.
\end{proposition}
\begin{proof}
For a codimension-$c$ algebraic cycle $Z$, the normalized ordinary Deligne cycle class is represented by $([Z],[Z],0)$ at level $c$. The map $\rho_{c,c-1}$ regards this same triple at level $c-1$. The degree-zero Friedlander--Gabber map is the quotient from rational equivalence to algebraic equivalence, and the resulting Lawson class has trace $[Z]$. Theorem~\ref{thm:mainmap} therefore gives the same triple. Both maps descend through rational equivalence, so agreement on algebraic cycles proves~\eqref{eq:comparisonzero}. The secondary assertion follows from the exact sequences at the two filtration levels. The cochain calculation is also valid with integral coefficients.
\end{proof}

\begin{example}[The Betti component in degree zero]\label{ex:chowbetti}
Take $X=\PP^3$, $c=2$, and a line $\ell\subset\PP^3$. Then
\[
 \CH^2(\PP^3)\simeq\Z,
 \qquad\HD^4(\PP^3,\Q(1))\simeq H^4(\PP^3,\Q).
\]
The image of $[\ell]$ has its nonzero usual Betti class. Thus a statement that $\delta_2\widehat A$ vanishes on \emph{all} higher Chow inputs cannot include degree zero without a homological-triviality hypothesis.
\end{example}

\subsection{The KLM current after lowering the filtration}
We now make a current-level comparison in every higher degree. Let
\[
 \square^k=(\PP^1\setminus\{1\})^k,
\]
with faces $z_j=0,\infty$, and write $z^c_{\R}(X,\bullet)_\Q$ for the rational cubical cycle subcomplex in good position with respect to the real cuts and all their faces. The subscript $\R$ refers to this position condition, not to real coefficients. The inclusion into the rational Bloch complex is a quasi-isomorphism by the real-position moving theorem cited above. Degenerate cubical cycles are quotiented out.

For $Z\in z^c_{\R}(X,k)_\Q$ with $\partial_BZ=0$, denote its normalized KLM currents by
\[
 (T_Z,\Omega_Z,\mathscr R_Z)
 \in \Icur_{X,\Q}^{M}\oplus F^c\Dcur_X^{M}\oplus\Dcur_X^{M-1}.
\]
Here $T_Z$ is the rational integration current of the real-cut chain, $\Omega_Z$ is the logarithmic form current, and $\mathscr R_Z$ is the logarithmic regulator current. The symbol $\mathscr R_Z$ is not the Lawson map $\Reg^X_{p,k}$ or the Lawson trace $R_f$. We fix the usual cubical and Poincar\'e-duality signs consistently so that
\begin{equation}\label{eq:KLMclosed}
 dT_Z=0,\qquad d\Omega_Z=0,
 \qquad d\mathscr R_Z=\Omega_Z-T_Z.
\end{equation}
In these conventions
\[
 r^{\mathrm{KLM}}_{c,k}([Z])=[(T_Z,\Omega_Z,\mathscr R_Z)].
\]
Compared with the universal unnormalized currents of \cite[Sections~5.3--5.5]{KLM}, the last two entries include the factor $(2\pi i)^{-k}$; the rational trace entry is untwisted. Fixed duality signs are absorbed in all three entries, not in one entry alone. The orientation at $k=0$ agrees with the algebraic cycle class.

\begin{lemma}[Filtration of the cut and logarithmic currents]\label{lem:KLMtype}
For an admissible degree-$k$ cycle, with $k\geq1$,
\begin{equation}\label{eq:KLMtype}
 T_Z\in F^{c-k}\Dcur_X^{M},\qquad
 \Omega_Z\in F^c\Dcur_X^{M},\qquad
 \mathscr R_Z\in F^{c-k}\Dcur_X^{M-1}.
\end{equation}
For an admissible degree-$(k+1)$ chain $W$, its cut current belongs to
$F^{c-k-1}\Dcur_X^{M-1}$.
\end{lemma}
\begin{proof}
The assertion is a bidegree calculation in the admissible current construction; it does not multiply arbitrary currents. On $X\times(\PP^1)^k$, the integration current of the closure of $Z$ has type $(c,c)$. A degree-$k$ real-cut current in the cube has components with nonnegative holomorphic and antiholomorphic degrees. Its product with $[Z]$, followed by pushforward to $X$, consequently has holomorphic degree at least $c-k$, since this pushforward lowers both indices by $k$.

The universal form $d\log z_1\wedge\cdots\wedge d\log z_k$ has type $(k,0)$, giving holomorphic degree $c$ after the same operation. Each term of the universal logarithmic current has the form
\[
 \varepsilon_j(2\pi i)^{j-1}\log z_j\,
 \delta_{T_{z_1}\cap\cdots\cap T_{z_{j-1}}}
 \wedge d\log z_{j+1}\wedge\cdots\wedge d\log z_k,
 \qquad \varepsilon_j\in\{1,-1\}.
\]
Its holomorphic degree is at least $k-j$, and is therefore nonnegative. After multiplying by $[Z]$ and projecting, its holomorphic degree is at least $c-k$ (indeed at least $c-j$ for this term). Admissibility is exactly what permits these operations in the KLM model. The same calculation can be made on a resolution of the algebraic support; proper pushforward has the stated bidegree shift. Limits of the admissible currents preserve the bound, because each $F^a\Dcur^b$ is a closed sum of bitype subspaces. Replacing $k$ by $k+1$ proves the final assertion. Constant normalizing factors do not change bitype.
\end{proof}

\begin{proposition}[Trace representative of the reduced ordinary regulator]\label{prop:KLMreduction}
For $q=c-k-1$ and every class $z\in\CH^c(X,k)_\Q$ represented by an admissible cycle $Z$,
\begin{equation}\label{eq:KLMreduction}
 \rho_{c,q}r^{\mathrm{Bloch}}_{c,k}(z)
 =\rho_{c,q}r^{\mathrm{KLM}}_{c,k}(z)
 =[(T_Z,T_Z,0)]\quad\text{in }\HD^M(X,\Q(q)).
\end{equation}
This expression is independent of the admissible representative of $z$.
\end{proposition}
\begin{proof}
The case $k=0$ is the ordinary cycle calculation. For $k\geq1$, Lemma~\ref{lem:KLMtype} gives $\mathscr R_Z\in F^{q+1}\subseteq F^q$. Thus $(0,\mathscr R_Z,0)$ is a degree-$(M-1)$ cochain in the lower-level cone, and
\[
 d_{\mathscr C}(0,\mathscr R_Z,0)
       =(0,d\mathscr R_Z,\mathscr R_Z).
\]
Subtract this boundary from $(T_Z,\Omega_Z,\mathscr R_Z)$ and use~\eqref{eq:KLMclosed}; the result is $(T_Z,T_Z,0)$. This proves~\eqref{eq:KLMreduction}, using the known equality~\eqref{eq:BlochKLM}.

For a direct check of descent, if $Z_1-Z_0=\partial_BW$ in the admissible complex, the cut-current identity gives $T_{Z_1}-T_{Z_0}=d(\varepsilon T_W)$, where $\varepsilon\in\{1,-1\}$ is fixed by the cubical convention. Lemma~\ref{lem:KLMtype} puts that primitive in $F^q$, and Lemma~\ref{lem:filteredcycle} identifies the two trace classes. The rational quasi-isomorphism of the admissible and Bloch complexes supplies such a rational admissible boundary whenever their classes agree. Degenerate cycles contribute zero to these cut and regulator currents in the normalized cubical complex.
\end{proof}
The proof actually removes the logarithmic entry already at level $c-k=q+1$. We use level $q$ here because an additional real parameter is needed for the general Lawson homotopies in Section~3. A vanishing logarithmic entry is not a vanishing Deligne class: the integration-current entry and its possible filtered primitives still determine the secondary value.

\begin{proposition}[Positive-degree Betti component of the ordinary regulator]\label{prop:BlochBetti}
For $k>0$ and $X$ smooth projective,
\[
 \delta_2r^{\mathrm{Bloch}}_{c,k}=0,
 \qquad\delta_2\rho_{c,q}r^{\mathrm{Bloch}}_{c,k}=0
 \quad\text{over }\Q.
\]
\end{proposition}
\begin{proof}
By~\eqref{eq:KLMclosed}, the Betti class of $T_Z$ belongs to
$H^M(X,\Q)\cap F^cH^M(X,\C)$. A rational class in $F^c$ also belongs to $\overline{F^c}$. But $M=2c-k<2c$, and a pure Hodge structure of weight $M$ has no component with both indices at least $c$. The intersection is therefore zero. Filtration reduction preserves the normalized Betti entry. This is the usual weight argument, also used in \cite[Section~5.6]{KLM}; no integral torsion vanishing is inferred.
\end{proof}
This is distinct from the assertion about $\widehat A$ at the beginning of the section, whose Betti component is defined through $\Phi\circ\mathrm{FG}$. Agreement at the Betti level alone would still not prove commutativity of~\eqref{eq:comparison-square} in Deligne cohomology.

\subsection{The precise comparison condition for the Friedlander--Gabber composite}
The cut current $T_Z$ and the trace of a representative of $\mathrm{FG}_{c,k}(z)$ are constructed in different realization models. The following elementary cone criterion states exactly what must be checked, rather than identifying those models by notation.

\begin{proposition}[Comparison by a trace difference]\label{prop:comparisoncriterion}
Let $z\in\CH^c(X,k)_\Q$. Choose an admissible representative $Z$ and a rational trace representative $R_f$ of $\mathrm{FG}_{c,k}(z)$. Define
\[
 \Delta_{c,k}(z)
  :=\widehat A_{c,k,\Q}(z)
       -\rho_{c,q}r^{\mathrm{Bloch}}_{c,k}(z).
\]
Then
\begin{equation}\label{eq:comparisondifference}
 \Delta_{c,k}(z)=[(R_f-T_Z,R_f-T_Z,0)].
\end{equation}
In particular its Betti component is $[R_f]-[T_Z]$. If these two Betti classes agree and $S$ is a rational locally integrally flat current with
\begin{equation}\label{eq:unfilteredprimitive}
 dS=R_f-T_Z,
\end{equation}
then, in the secondary quotient,
\begin{equation}\label{eq:comparisonobstruction}
 \Delta_{c,k}(z)=[\bar S]\quad\text{in}\quad
 \frac{H^{M-1}(\Dcur_X^\bullet/F^q)}
      {\im H^{M-1}(X,\Q)}
 \simeq J_q^{M-1}(X;\Q).
\end{equation}
Thus the square~\eqref{eq:comparison-square} commutes on $z$ if and only if the Betti classes agree and $[\bar S]$ is a rational period. A sufficient condition is
\begin{equation}\label{eq:filteredcomparison}
 R_f-T_Z=dS,\qquad
 S\in\Icur_{X,\Q}^{M-1}\cap F^q\Dcur_X^{M-1}.
\end{equation}
\end{proposition}
\begin{proof}
Theorem~\ref{thm:mainmap} gives $\widehat A(z)=[(R_f,R_f,0)]$ and Proposition~\ref{prop:KLMreduction} gives the other trace triple, proving~\eqref{eq:comparisondifference}. Both traces belong to $F^{q+1}$. Under~\eqref{eq:unfilteredprimitive}, their difference is a closed current in $\Icur_{X,\Q}^M$ with zero rational cohomology class. The quotient-cone calculation in Proposition~\ref{prop:AJ} applies to it and gives~\eqref{eq:comparisonobstruction}. Changing $S$ changes $[\bar S]$ by rational cohomology. The exact sequence~\eqref{eq:Deligneexact} proves the necessary and sufficient criterion. Under the stronger condition~\eqref{eq:filteredcomparison},
\[
 d_{\mathscr C}(S,S,0)=(R_f-T_Z,R_f-T_Z,0),
\]
which proves equality directly. Choices of representatives do not alter $\Delta$, because the two maps defining it are already well-defined.
\end{proof}
Under Poincar\'e duality the secondary difference in~\eqref{eq:comparisonobstruction} is evaluated by the current pairing $\langle S,\omega\rangle$ on
\[
 F^{p+k+2}H^{2p+k+1}(X,\C),\qquad p=n-c,
\]
modulo the full rational period group. The ordinary regulator is tested on the larger space $F^{p+1}H^{2p+k+1}(X,\C)$; the map $\rho_{c,q}$ restricts that functional to $F^{p+k+2}$. This follows from duality of the two quotients in~\eqref{eq:rhoJac}, not from replacing one Tate lattice by another without normalization.

Condition~\eqref{eq:filteredcomparison} holds, in particular, for any class for which a common geometric realization has been constructed with $R_f=T_Z$, or with the two traces joined by an admissible prism of algebraic $p$-cycle families. In the latter case Proposition~\ref{prop:type} supplies the required $F^q$ bound. This is a conditional comparison statement: existence of such a common realization must be proved for the class under consideration.

Friedlander--Gabber's Construction~4.1 is a zigzag in the derived category, involving topological spaces of cycles on algebraic simplices and homotopy equivalences. The rational real-position moving theorem for the KLM complex does not identify that zigzag with a Lawson family over the same real cube, and does not alone supply~\eqref{eq:filteredcomparison}. We do not prove that condition for the full Friedlander--Gabber map here. Accordingly, the equality
\[
 \widehat A_{c,k,\Q}
       =\rho_{c,c-k-1}\circ r^{\mathrm{Bloch}}_{c,k}
       \qquad(k>0)
\]
is not asserted on all of $\CH^c(X,k)_\Q$. The proved trace reduction~\eqref{eq:KLMreduction} must not be read as a proof of this additional identification.

\subsection{Decomposable unit classes: an unconditional positive-degree comparison}
There is a useful class of positive-degree inputs for which both sides of the comparison can be determined without the remaining realization argument. For $c\geq3$, let
\[
 \CH^c(X,1)_{\mathrm{dec},\Q}
  :=\im\bigl(\CH^{c-1}(X)_\Q\otimes_\Q(\C^*\otimes_\Z\Q)
          \longrightarrow\CH^c(X,1)_\Q\bigr).
\]
Here a generator is written $W\cdot\{u\}$ with $u\in\C^*$ a constant unit and $W$ a codimension-$(c-1)$ cycle; the tensor notation refers to the usual product of higher Chow groups.

\begin{proposition}[Vanishing on constant-unit products]\label{prop:decomposable}
For $3\leq c\leq n$ and $q=c-2$,
\[
 \widehat A_{c,1,\Q}|_{\CH^c(X,1)_{\mathrm{dec},\Q}}
 =\bigl(\rho_{c,c-2}r^{\mathrm{Bloch}}_{c,1}\bigr)
            |_{\CH^c(X,1)_{\mathrm{dec},\Q}}=0.
\]
\end{proposition}
\begin{proof}
First consider the Friedlander--Gabber realization of an integral generator $W\cdot\{u\}$. By linearity it suffices to take $W$ irreducible. Choose a smooth projective resolution $g:\widetilde W\to W\hookrightarrow X$, and put $d=\dim\widetilde W=n-c+1\geq1$. In higher Chow groups,
\[
 W\cdot\{u\}=g_*(\{u\}),\qquad
 \{u\}\in\CH^1(\widetilde W,1)\simeq\C^*.
\]
The last identification is Bloch's codimension-one calculation. Friedlander's divisor computation gives
\[
 L_{d-1}H_{2d-1}(\widetilde W)
       \simeq H_1(\Pic^0(\widetilde W),\Z),
\]
a finitely generated free abelian group \cite[Theorem~4.6]{F}. Since $\widetilde W$ is connected and projective, its global units are $\C^*$, a divisible abelian group. A homomorphism from a divisible group to a finitely generated free abelian group is zero. Therefore $\mathrm{FG}^{\Z}_{1,1}(\{u\})=0$ on $\widetilde W$. This argument is made before rationalization: divisibility alone would not force a $\Q$-linear map between rational vector spaces to vanish. Proper pushforward commutes with the comparison of \cite[Construction~4.1]{FG}, since it commutes with the cycle-space maps in its defining zigzag. It follows that $\mathrm{FG}^{\Z}_{c,1}(W\cdot\{u\})=0$ on $X$. Rationalizing and applying $\Reg$ proves the first vanishing.

For the second, the product and logarithmic formulas for the ordinary regulator \cite[Sections~4 and~5.5]{KLM} represent $r^{\mathrm{Bloch}}_{c,1}(W\cdot\{u\})$ in the normalized cone by
\[
 (0,0,a),\qquad
 a=\frac{\log u}{2\pi i}\,[W].
\]
A branch change adds an integral multiple of the rational cycle class $[W]$, so the Deligne class is well-defined. The current $a$ is closed and of type $(c-1,c-1)$. In the lower-level cone $a\in F^{c-2}$, and
$d_{\mathscr C}(0,a,0)=(0,0,a)$. This proves the second vanishing.
\end{proof}
The proof uses established codimension-one, divisor, and product calculations. It proves equality on this decomposable subgroup, not on every degree-one higher Chow class.

\begin{example}[Information lost under filtration reduction]\label{ex:unitregulator}
On $X=\PP^3$, let $h\in\CH^1(\PP^3)$ be the hyperplane class and set
\[
 z=h^2\cdot\{2\}\in\CH^3(\PP^3,1).
\]
Write $h_B$ for the normalized rational Betti class of a hyperplane. The ordinary regulator has value
\begin{equation}\label{eq:unitexample}
 r^{\mathrm{Bloch}}_{3,1}(z)
   =\left[\frac{\log 2}{2\pi i}\,h_B^2\right]
   \in\HD^5(\PP^3,\Q(3))\simeq\C/\Q,
\end{equation}
which is nonzero because $\log 2/(2\pi i)$ is a nonzero purely imaginary number and is not rational. Here $H^4$ has type $(2,2)$, so $F^3H^4=0$. At the reduced level, however,
\[
 \HD^5(\PP^3,\Q(1))=0,
 \qquad \widehat A_{3,1,\Q}(z)=0.
\]
The vanishing follows either from the target or from Proposition~\ref{prop:decomposable}. This is a calculation with a standard decomposable higher Chow class, not a new nonvanishing theorem for higher Chow groups. It shows that $\widehat A$ cannot retain the entire ordinary regulator, regardless of the unresolved general commutativity of~\eqref{eq:comparison-square}.
\end{example}

\section{Conclusions and remaining questions}\label{sec:conclusion}
The Lawson--Deligne construction has an integral trace-current description, and its rational secondary part is governed by the Hodge quotient~\eqref{eq:Jintro}. The chain/current sign is fixed in Section~\ref{sec:cone}, and Proposition~\ref{prop:AJ} distinguishes geometric fillings from cohomological primitives. The two trace estimates~\eqref{eq:two-estimates} have different roles: $F^{q+1}$ controls the representative, whereas $F^q$ controls the homotopy that makes its class well-defined. Ordinary algebraic cycle operations, rather than an unrestricted product of shifted Lawson regulators, provide the functorial operations used in the applications.

The period calculations keep the order of choices explicit. For fixed geometric data on $Y$, one excludes exceptional elliptic moduli. For prescribed moduli, one instead varies the mirror-quintic fiber and applies separation over the fixed countable coefficient space. The latter argument gives the exact dimension of the specified subspace $\LoopIm_{p,k}$, including the relations forced by repeated quadratic factors. These are not calculations of the entire image $\Img_{p,k}$. The blow-up applications preserve the specified subspaces through the known exceptional decompositions; they do not assert a new Lawson blow-up theorem or infinite rank on one fixed rational variety.

The higher Chow comparison has three distinct levels. The equality of the ordinary Bloch and rational KLM regulators is a published result. Proposition~\ref{prop:KLMreduction} proves that their filtration reduction is represented by the cut-current triple $(T_Z,T_Z,0)$. To identify that triple with the composite through Lawson homology in arbitrary positive degree, one must compare the cut current with the Friedlander--Gabber trace. Proposition~\ref{prop:comparisoncriterion} gives the precise difference class and its vanishing criterion. Equality is proved here in degree zero and, in degree one, on constant-unit decomposable classes. No general equality or general inequivalence of the two reduced maps is inferred from their different constructions.

The remaining realization problem is to construct a comparison of the relevant cycle complexes whose homotopies have traces in $F^{c-k-1}$, or to exhibit a class with a nonzero comparison difference. A chain homotopy only after forgetting the filtration would not suffice. Independently, the geometric examples raise the questions of effective nonvanishing at a prescribed algebraic parameter and of the full regulator image on a fixed threefold or fixed rational variety. These questions are not settled by a countable avoidance argument or by the finite-dimensional period-monomial calculation. The applications in Sections~\ref{sec:PF}--\ref{sec:blowup} concern Lawson classes directly and do not assume that those classes lie in the image of $\mathrm{FG}$ or that the comparison square~\eqref{eq:comparison-square} commutes.

%\clearpage

\end{document}